\documentclass[12pt]{amsart}

\usepackage{amssymb,latexsym,amsthm, amsmath, comment, fullpage}
\usepackage[all]{xy}
\usepackage{graphicx, comment,  tikz-cd, libertine}

\ifx\pdfoutput\undefined
\usepackage{hyperref}
\else
\usepackage[pdftex,colorlinks=true,linkcolor=blue,urlcolor=blue]{hyperref}
\fi

\theoremstyle{plain}
\newtheorem{theorem}{Theorem}[section]

\newtheorem{corollary}{Corollary}[section]
\newtheorem{proposition}{Proposition}[section]

\newtheorem{lemma}{Lemma}[section]
\theoremstyle{definition}
\newtheorem{definition}{Definition}[section]
\newtheorem{remark}{Remark}[section]

\newtheorem*{ack}{Acknowledgements}

\newcommand{\vertiii}[1]{{\left\vert\kern-0.25ex\left\vert\kern-0.25ex\left\vert #1 
    \right\vert\kern-0.25ex\right\vert\kern-0.25ex\right\vert}}
   
\begin{document}

\title[On transversal H\"older regularity for flat Wieler solenoids]
      {On transversal H\"older regularity for flat Wieler solenoids}
      \author{Rodrigo Trevi\~no}
      \address{Department of Mathematics, The University of Maryland, College Park, USA}
      \email{rodrigo@trevino.cat}
      \date{\today}
      \begin{abstract}
        This paper studies various aspects of inverse limits of locally expanding affine linear maps on flat branched manifolds, which I call flat Wieler solenoids. Among the aspects studied are different types of cohomologies, the rates of mixing given by the Ruelle spectrum of the hyperbolic map acting on this space, and to solutions of the cohomological equation in primitive substitution subshifts for H\"older functions. The overarching theme is that considerations of $\alpha$-H\"older regularity on Cantor sets go a long way.
      \end{abstract}
      \maketitle
      \section{Introduction and statement of results}
      In his book on thermodynamic formalism \cite{ruelle:book}, D. Ruelle introduced the concept of Smale spaces, which intended to axiomatically generalize the concept of a smooth hyperbolic map to the setting of compact metric spaces. This general framework was a unified theory under which many types of chaotic systems could be included. Some of the systems included under this umbrella are Anosov systems, basic sets of Axiom A systems, mixing shifts of finite type, hyperbolic toral automorphisms, etc. One of the many features shared by these systems are the existence of stable and unstable sets, generalizing the concept of stable and unstable manifolds found in some smooth systems. Later, Wieler \cite{Wieler} characterized Smale spaces with zero dimensional stable sets as inverse limits satisfying certain conditions. 

      This paper studies various aspects of inverse limits of locally expanding affine linear maps on flat branched manifolds, which I call \textbf{flat Wieler solenoids} since they are a subclass of the types of Smale spaces classified by Wieler. They are spaces which are in a sense of intermediate type: one one extreme, Anosov maps are defined on smooth manifolds whereas mixing subshifts of finite type are defined on totally disconnected sets. Flat Wieler solenoids lie somewhere between these two extremes and it is this intermediate position that makes them particularly interesting. Although this is a subclass of Wieler's more general classification, it is a class with a lot of rich examples. In particular, it includes the tiling spaces of self-similar tilings, or tilings built using a substitution rule. They are also related to primitive substitution systems in symbolic dynamics.

      Let me be more precise (all the following terms are defined in \S \ref{sec:intro}): let $\Gamma$ be a flat branched manifold of dimension $d>0$ and $\gamma:\Gamma\rightarrow \Gamma$ be a locally expanding affine linear map. This means that $\gamma$ has constant derivative $D_\gamma\in GL(d,\mathbb{R})$ outside the branch points of $\Gamma$, with smallest eigenvalue $\lambda_0$ satisfying $|\lambda_0|>1$. The inverse limit of the pair $(\Gamma,\gamma)$ is
      $$\Omega_\gamma := \left\{ (z_0,z_1,z_2,\dots)\in \Gamma^\infty: z_i = \gamma(z_{i+1})\mbox{ for all }i\geq 0\right\}.$$
      This is a flat Wieler solenoid. It has a local product structure of the form $B_\varepsilon\times \mathcal{C}$ where $B_\varepsilon$ is the $\varepsilon$-ball of dimension $d$ and $\mathcal{C}$ is a Cantor set. It becomes a Smale space when paired with the ``hyperbolic'' map $\Phi:\Omega_\gamma\rightarrow \Omega_\gamma$ defined by applying $\gamma$ to every coordinate of $z\in\Omega_\gamma$. Under this map, the Cantor sets in the local structure become stable sets and the Euclidean components become the unstable sets. Under suitable conditions, $\Omega_\gamma$ is a foliated space with each leaf dense in $\Omega_\gamma$ and homeomorphic to $\mathbb{R}^d$.

      Although $\Omega_\gamma$ is not a manifold, it still has a rich topological and geometric structure. First, the \v Cech cohomology can be computed as the direct limit of the cohomologies of $\Gamma$ under the induced map $\gamma^*$. Another way to recover this topological information \` a la de Rham is to consider the set of functions with the highest possible regularity on $\Omega$, consider differential forms with coefficients in this space of functions, and then compute the cohomology of this complex of forms defined by the natural leafwise de Rham differential operator \cite{KP:RS}. The space of functions which serves this role are the leafwise $C^\infty$, transversally locally constant functions, denoted by $C^\infty_{tlc}(\Omega)$. These functions have the highest regularity in the smooth direction ($C^\infty$) as well as the highest regularity in the transverse, totally disconnected direction (locally constant).

      One of the motivations behind the paper is a lack of good spaces of functions on $\Omega_\gamma$ of intermediate regularity. The usual tricks afforded by Sobolev spaces on smooth manifolds are not immediately available in this setting since it is not clear how they can afford an accounting of the regularity of a function in the totally disconnected direction. I propose that looking at the H\"older regularity in the totally disconnected direction can get us started in working with spaces of functions with intermediate regularity. Unlike Euclidean spaces, the space of $\alpha$-H\"older functions on a Cantor set is non-trivial for all $\alpha>0$ and thus $\alpha$-H\"older functions do supplant the notion of regularity in this setting.
      
      To get a grasp on a good space of smooth functions on $\Omega$, in \S \ref{subsec:functions} I define the set $\mathcal{S}^r_\alpha$ of functions on $\Omega$ which are roughly described as functions which are $C^r$ in the leaf direction and $\alpha$-H\"older in the transverse direction. This is how one can handle different degrees of regularity in the smooth and non-smooth directions of the space $\Omega$, respectively. It will be shown that $\mathcal{S}^r_\alpha$ is in a sense dense in the space $C^r_\alpha$ of leafwise $C^r$ functions which are $\alpha$-H\"older in the transverse direction in up to $r$ derivatives.

      The concept of $\alpha$-H\"older functions implicitly makes reference to a metric on $\Omega$ and so a function which is $\alpha$-H\"older in one metric may be $\alpha'$-H\"older in another metric for some $\alpha'\neq\alpha$. In these types of spaces, there is a natural metric in reference to which all statements will be made. One of the properties of this metric is that, for any leafwise first order differential operator $\partial$ and $f\in\mathcal{S}^r_\alpha$, the derivative satisfies $\partial f \in \mathcal{S}^{r-1}_{\alpha+1}$. In other words, leafwise differentiation reduces leafwise regularity, but \emph{increases} transverse regularity. It seems to me like this has gone unnoticed, and there is much to gain from this observation.

      Let $H^*_{r,\alpha}$ be the cohomology of the complex of leafwise smooth forms with coefficients in $\mathcal{S}^r_\alpha$
      . The first result gives a lower bound on the regularity of functions in this complex above which the usual real cohomology of $\Omega$ can be recovered.
      \begin{theorem}
        \label{thm:coh}
        Let $\Omega$ be a flat Wieler solenoid and $H^*_{r,\alpha}(\Omega)$ the cohomology of tangential forms on $\Omega$ with coefficients in $\mathcal{S}^r_\alpha(\Omega)$. If $r\in\mathbb{N}$ and $\alpha>1$ then
        $$H^*_{r,\alpha}(\Omega)\cong \check H^*(\Omega;\mathbb{R}).$$
      \end{theorem}
      This theorem should remind one of de Rham regularization. Indeed, part of the proof uses de Rham regularization (to find bounds for $r$). However, more needs to be done to ensure that $\alpha>1$ guarantees that the cohomology is finite dimensional. The first immediate application of this theorem is for the speed of ergodicity of functions in $\mathcal{S}^r_\alpha$ for $r\in\mathbb{N}$ and $\alpha>1$. This essentially follows from the arguments in \cite{ST:SA} and it is discussed in \S \ref{subsec:deviations}. It improves previous results on deviations of ergodic integrals in that it increases the set of functions for which the deviation results hold.

      As mentioned above, there is a self-homeomorphism $\Phi:\Omega_\gamma\rightarrow \Omega_\gamma$ which makes a flat Wieler solenoid a Smale space. This map preserves an absolutely continuous probability measure $\mu$ and is topologically mixing \cite[Proposition 3.1]{AP}. The second application of the careful study of transverse H\"older regularity is to the speed of mixing of the map $\Phi:\Omega\rightarrow \Omega$. In order to do this, the notion of the Ruelle spectrum has to be defined.
      \begin{definition}
        Let $\Phi:\Omega\rightarrow \Omega$ be a map preserving a probability measure $\mu$ and $\mathcal{F}$ a space of bounded functions on $\Omega$. Let $I$ be a finite or countable set , $\Lambda = \{\lambda_i\}_{i\in I}$ a set of complex numbers with $|\lambda_i|\in (0,1]$ such that for any $\varepsilon>0$ there are only finitely many $i$ such that $|\lambda_i|>\varepsilon$, and let $\{N_i\}_{i\in I}$ be non-negative integers. Then $\Phi$ has \textbf{Ruelle spectrum $\Lambda$} with Jordan block dimension $\{N_i\}$ on $\mathcal{F}$ if, for any $f,g\in \mathcal{F}$ and $\varepsilon>0$, there is an asymptotic expansion
          $$\int_{\Omega} f\cdot g\circ \Phi^n\, d\mu = \sum_{|\lambda_i|\geq\varepsilon}\sum_{j\leq N_i} \lambda_i^nn^jc_{i,j}(f,g) + o(\varepsilon^n)$$
          where $c_{i,j}(f,g)$ are nonzero bilinear functions of $f$ and $g$ of finite rank.
      \end{definition}
      Note that these asymptotics give precise information on the rates of mixing. Moreover, if $f\in \mathcal{F}$ is an eigenfunction for $\Phi$ with eigenvalue $\nu$ and the essential spectrum for $\Phi^*$ is reduced to $\{0\}$, then $\nu\in \Lambda$. Thus the search for the Ruelle spectrum reduces to the search of generalized eigenfunctions for the pullback operator defined by $\Phi$, called the transfer operator $\mathcal{L} = \Phi^*$, on a good space of functions $\mathcal{F}$. To find a good $\mathcal{F}$, the use of the so-called anisotropic Banach spaces will be employed.
      
      Computing the Ruelle spectra of systems has become fashionable in the last half-decade, especially through the use of anisotropic Banach spaces. In \cite{FGL:anosov}, using transfer operator techniques and anisotropic Banach spaces, the authors noticed that the Ruelle spectrum for linear pseudo-Anosov maps on Riemann surfaces is composed entirely from cohomological information. Using different techniques, this was soon reproved by \cite{forni:ruelleLectures}, and that point of view was extended to the case of non-linear pseudo-Anosov actions on surfaces \cite{forni:nonlinear}. Those works were followed by \cite{BKL:locating} and the recent PhD dissertation of D. Galli -- both using transfer operators and anisotropic Banach spaces -- where the focus has been to extract resonances from the cohomology spectrum in a larger class of non-linear systems. They concluded that the Ruelle spectrum \emph{contains} cohomological information, but is not necessarily made up exclusively of cohomological information. The results here are of that type.
      
      I will define anisotropic Banach spaces $\mathcal{B}^{r,\alpha}_m$ as the completion of the space of tangential $m$-forms with coefficients in $\mathcal{S}_\alpha^r$ with respect to an anisotropic norm, and then study the spectrum of the transfer operator on these spaces. The spaces of functions for which part of the Ruelle spectrum can be identified will be $\mathcal{B}_0^{\infty,\alpha} = \bigcap_{r>0}\mathcal{B}_0^{r,\alpha}$ for $\alpha$ large enough.

      Let $h_{top}$ be the topological entropy of $\gamma:\Gamma\rightarrow \Gamma$ and $\chi_-$ the smallest Lyapunov exponent of $\gamma$. Let $\sigma^-$ be the set of eigenvalues $\nu$ of $\Phi^{-1*}:\check H^d(\Omega;\mathbb{R})\rightarrow \check H^d(\Omega;\mathbb{R})$ which satisfy $\log|\nu|< \chi_- - h_{top}$. Note that when $d=1$ this is the set of all contracting eigenvalues.
      \begin{theorem}
        \label{thm:ruelle}
        Let $\Phi:\Omega\rightarrow \Omega$ be the topologically mixing map on the solenoid which preserves the absolutely continuous probability measure $\mu$. For $\alpha>\frac{h_{top}}{\chi_-}$ and $r\in\mathbb{N}$:
        \begin{enumerate}
        \item if $d=1$, then the set of eigenvalues for $\mathcal{L} = \Phi^{-1*}$ acting on $\mathcal{B}_0^{r,\alpha}$ contains $\sigma^-\setminus \{e^{-h_{top}}\}$. In addition, if $\nu$ is an eigenvalue in $\mathcal{B}_0^{r,\alpha}$ and $k<\alpha-\frac{h_{top}}{\chi_-}$, then $e^{-kh_{top}}\nu$ is an eigenvalue in $\mathcal{B}_0^{r+k,\alpha-k}$. It follows that if $\mathcal{F}:= \displaystyle \bigcap_{\alpha>0,r>0}\mathcal{S}^r_\alpha$, then the Ruelle spectrum for functions in $\mathcal{F}$ contains the set of numbers of the form $e^{-kh_{top}}\nu$ with $\nu\in \sigma^-\setminus \{e^{-h_{top}}\}$ and $k\in \mathbb{N}$.
        \item If $d=2$, then the set of eigenvalues for $\mathcal{L} = \Phi^{-1*}$ acting on $\mathcal{B}_0^{r,\alpha}$ contains $\sigma^-\setminus \{e^{-h_{top}}\}$. If $\mathcal{S}^{\infty}_\alpha:= \displaystyle\bigcap_{r>0}\mathcal{S}^r_\alpha$, then the Ruelle spectrum for functions in $\mathcal{S}^{\infty}_\alpha$ contains the set $\sigma^-\setminus \{e^{-h_{top}}\}$ .
        \end{enumerate}
      \end{theorem}
      \begin{remark}
        Some remarks:
        \begin{enumerate}
        \item Since the tiling spaces of self similar tilings are flat Wieler solenoids \cite{AP} of this type, Theorem \ref{thm:ruelle} gives the rate if mixing of the substitution rule on these spaces. This can be interpreted as the rates of decay of correllations of the different scales of the tiling.
        \item Let me connect the one-dimensional case with the Pisot conjecture. The hypotheses of the \emph{homological} Pisot conjecture assert that $\sigma^-_1\setminus \{e^{-h_{top}}\}=\varnothing$ (see the survey \cite{Pisot:survey}). Thus, under the hypothesis of the homological Pisot conjecture I find no obstruction to having super exponential decay of correlations, which is a feature of algebraic systems. Thus, showing that $\sigma^-\setminus \{e^{-h_{top}}\}$ is the entire spectrum would be very valuable.
        \item Examples of two-dimensional solenoids for which $\sigma^-_1\setminus \{e^{-h_{top}}\}\neq \varnothing$ include the tiling spaces of self-similar tilings which are weakly mixing under the translation action along leaves, such as the Godrèche-Lançon-Billard tiling, and one of Danzer's sevenfold tilings. See \cite[\S 6.5]{BG:book} for more details.
        \end{enumerate}
      \end{remark}
      
      Finally, another application of the study of transverse H\"older regularity in these spaces is to the solution of the cohomological equation in the setting of primitive substitution subshifts.
      Let $\mathcal{A}$ be a finite set and let $\varrho$ be a primitive substitution rule on $\mathcal{A}$ (all of these terms are defined in \S \ref{sec:substitution}). Let $\sigma:X_\varrho\rightarrow X_\varrho$ be the minimal subshift defined by this substitution rule. Let $H_\alpha(X_\varrho)$ be the space of $\alpha$-H\"older functions on $X_\varrho$ (this is with respect to some natural ultrametric; the following results are stated with respect to a specific, natural ultrametric). Let $H^0_\alpha(X_\varrho)$ be the quotient of $H_\alpha(X_\varrho)$ with respect to the equivalence relation $f\sim g$ if there exists a $u\in H_{\alpha-2}(X_\varrho)$ such that $f-g = u\circ \sigma- u$. This is the \textbf{$\alpha$-H\"older cohomology of $X_\varrho$}.
      \begin{theorem}
        \label{thm:livsic}
        Let $\mathcal{A}$ be a finite set and let $\varrho$ be a primitive substitution rule on $\mathcal{A}$. If $\alpha>2$ then $H^0_\alpha(X_\varrho)$ is finite dimensional. 
        That is, if $\alpha>2$ there are finitely many obstructions to finding a solution $u\in H_{\alpha-2}(X_\varrho)$ to the cohomological equation $f = u\circ \sigma-u$ for $f\in H_\alpha(X_\varrho)$. 
      \end{theorem}
      \begin{remark}
        Some remarks:
        \begin{enumerate}
        \item The locally constant, \emph{integral} cohomology has been thoroughly studied for minimal Cantor systems. More precisely, the structure of the set of equivalence classes of $C(X,\mathbb{Z})$ up to coboundaries is a fundamental invariant in the theory of orbit equivalence for Cantor minimal systems \cite{GPS}. 
          However, as far as I know, the problem of solving the cohomological equation for varrying degrees of regularity on Cantor sets has not been considered before.
        \item The ideas leading up to the theorem above do not only hold for substitution systems. In fact, with the right use of Oseledets theorem, I expect the theorem above to hold for a large class of minimal Cantor systems, including a large class of so-called $S$-adic systems. This will be pursued in a future paper.
        \item It is unclear whether a loss of regularity of order $2$ is optimal. Although a loss of 1 may be necessary, a loss of 2 may be a feature of the way the theorem is proved. I would like to see a proof of this statement which does not rely on embedding $X_\varrho$ into a solenoid as it does here.
        \end{enumerate}
      \end{remark}      
      This paper is organized as follows. Background material is covered in \S \ref{sec:intro}, including the function spaces $\mathcal{S}^r_\alpha$, which as far as I know are new. Section \S \ref{sec:cohomology} deals with proving that the $r,\alpha$ cohomology are isomorphic to the usual real cohomologies (Theorem \ref{thm:coh}) for $r,\alpha$ large enough. Section \S \ref{sec:ruelle} is devoted to the study of the Ruelle spectrum: anisotropic Banach spaces of forms are introduced and the action of the transver operator on them is studied leading to Theorem \ref{thm:ruelle}. Finally, in \S \ref{sec:substitution}, the solutions of the cohomological equation on $X_\varrho$ are studied by relating them to solutions of the cohomological equation on one dimensional solenoids, proving Theorem \ref{thm:livsic}.
      \begin{ack}
        I am in debt to Daniele Galli, Scott Schmieding, Giovanni Forni, and Oliver Butterley for very useful discussions during the preparation of this paper. This work was partially supported by grant 712227 of the Simons Foundation, as well as DMS grant 2143133 from the National Science Foundation.
      \end{ack}
      \section{Wieler solenoids of flat branched manifolds}
      \label{sec:intro}
      Let $\Gamma$ be a connected, flat branched manifold of dimension $d$. This means that $\Gamma$ is obtained by gluing several polytopes of dimension $d$ along their faces in such a way that every $k$-dimensional face meets another of the same dimension. This branched manifold has a natural CW structure and flat metric; denote by $I^k$ its $k^{th}$-skelleton. Let $B$ be the branching set, that is the set of points $x$ such that a small ball around $x$ is not homeomorphic to an open euclidean ball. If $B=\varnothing$ then $\Gamma$ is a manifold. The focus here will be in cases where $B\neq \varnothing$. Setting $B_k := (I^k\setminus I^{k-1})\cap B$, it is worth pointing out that every point on $B_k$ is a $k$-dimensional flat  manifold and, as such, it has well defined tangent and cotangent spaces.

      Let $\gamma:\Gamma\rightarrow \Gamma$ be a locally expanding and surjective map with constant compatible derivative. This means that the derivative map is non-singular and on $I^d\setminus I^{d-1}$ the derivative $D_\gamma$ can be identified with some $D_\gamma\in GL(d,\mathbb{R})$. Let $\lambda_1,\dots,\lambda_d>1$ be the eigenvalues of $D_\gamma$ and set $\lambda_0 = \min_i|\lambda_i|$ and $\lambda = \mathrm{det}\, D_\gamma$.

      The \textbf{solenoid defined by $\gamma$} is the space
      \begin{equation}
        \label{eqn:Scoords}
  \Omega=\Omega_\gamma=\{\bar{z} = (z_0,z_1,\dots)\in \Gamma^\infty:\gamma(z_{i+1}) = z_i\mbox{ for all }i\geq 0\}.
  \end{equation}
      It can also be defined as an inverse/projective limit: the inverse limit
      \begin{equation}
        \label{eqn:invLim}
        \Omega_\gamma = \lim_{\leftarrow}(\Gamma,\gamma)
      \end{equation}
      can easily be seen to match the first definition. This is a \textbf{flat Wieler solenoid} since it is part of Wieler's classification of Smale spaces with totally disconnected stable sets (this structure will be examined below). The $\gamma$-solenoid comes equipped with a probability measure which is compatible with the inverse limit structure. To describe this more precisely, denote by $\pi_k:\bar{z}\mapsto z_k$ the projection onto the $k^{th}$ coordinate, and denote by $\Gamma_k$ to be the $k^{th}$ copy of $\Gamma$: $\Gamma_k = \pi_k(\Omega)$. Let $\mu_k$ be the normalized Lebesgue measure on $\Gamma_k$ induced from the flat metric. Since $\gamma$ preserves Lebesgue measure, it follows that $\gamma_*\mu_k = \mu_{k-1}$. Thus if we equip $\Omega$ with $\mu := \otimes_k\mu_k$, we have that $\pi_{k*}\mu = \mu_k$ for every $k$.

      Two assumptions need to be made on the pair $(\Gamma,\gamma)$:
      \begin{enumerate}
      \item The map $\gamma:\Gamma\rightarrow \Gamma$ is \textbf{primitive}: there exists a $K>0$ such that for any two faces $F_1,F_2\subset \Gamma$ there is a point $x\in F_1$ such that $\gamma^K(x)\in F_2$.
      \item The map $\gamma:\Gamma\rightarrow \Gamma$ \textbf{forces the border}. Although I will not discuss what is means here, an implication of this assumption will be pointed out below.
      \item The map $\gamma$ is \textbf{recognizable}: roughly speaking, all $d$-cells have distinct images.
      \end{enumerate}
      
      The solenoid $\Omega$ has a local product structure of $B_\varepsilon(0)\times\mathcal{C}$, where $B_\varepsilon(0)$ is an $\varepsilon$-ball in $\mathbb{R}^d$ and $\mathcal{C}$ is a Cantor set. Indeed, if $z = (z_0,z_1,\dots)\in \Omega$, then points close to $z$ come from either varying the $z_0$ coordinate by a small amount (this is parametrized by $B_\varepsilon$), or by varying in $\pi_0^{-1}(z_0)$, which amounts to picking one of (potentially several) points in $\gamma^{-1}(z_0)$, then one of (potentially several) points in the $\gamma$-preimage of that point, and so on. It follows that since $\lambda>1$, this sequence of choices naturally gives a Cantor set. 

      There are two complementary dynamical systems defined on $\Omega$. First, there is a self-homeomorphism $\Phi:\Omega\rightarrow \Omega$ which preserves the special measure $\mu$. In coordinates, the map is defined by
      $$\Phi:(z_0,z_1,z_2,\dots)\mapsto (\gamma(z_0),z_0,z_1,\dots)$$
      which, by construction, preserves the measure $\mu$. The inverse is obtained by deleting the first coordinate.
      
      For $z\in \Omega$ define its $k^{th}$ transversal set as
      $$C_k^\perp(z) := \left\{z'\in \Omega: z_i' = z_i\mbox{ for all }i\leq k \right\}.$$
      The sets $C_0^\perp(x)$ are all ultrametric sets which will endowed with the metric
      \begin{equation}
        \label{eqn:metric}
        d(y,z) = d_x(y,z) := \lambda_0^{-k(y,z)},
      \end{equation}
      where $k(y,z)$ is the smallest integer $i$ such that $y_i\neq z_i$; equivalently, the smallest integer $i$ such that $C_i^\perp(y)\neq C_i^\perp(z)$. These transversal sets can be seen to be the local stable set of $x$: if $y\in C_0^\perp(x)$, then $d(\Phi^n(x), \Phi^n(y))\rightarrow 0$ as $n\rightarrow \infty$. The local unstable sets can also be seen, in the coordinates $B_\varepsilon\times \mathcal{C}$ around $z$, to be the Euclidean balls $B_\varepsilon\times \{c\}$ for $c\in\mathcal{C}$.
      
      Secondly, there is the $\mathbb{R}^d$ action on $\Omega$ given by translating along unstable sets. More precisely, for $z\in \Omega$ with $z_0=\pi_0(z)$ not in a branch of $\Gamma$ and $t\in \mathbb{R}^d$ of small norm, the translation of $z = (z_0,z_1,z_2,\dots)$ by $t$ is      
      \begin{equation}
        \label{eqn:Rdaction}
        \varphi_t(z):= (z_0-t, z_1-D_\gamma^{-1}t, z_2-D_\gamma^{-2}t,\dots),
      \end{equation}
      which is seen to preserve the condition (\ref{eqn:Scoords}). The assumption that $(\Gamma,\gamma)$ forces the border implies that this extends to an action of $\mathbb{R}^d$ on $\Omega$; the primitivity condition implies that this action is minimal, that is, every orbit is dense. Finally, the recognizability condition implies that the action is free and so every orbit is homeomorphic to $\mathbb{R}^d$. Under these conditions, the action of $\mathbb{R}^d$ is uniquely ergodic \cite[Theorem 3.1]{solomyak:SS}, where the unique invariant probability measure is $\mu$.
      
      Since the solenoid $\Omega$ has a local product structure of $B_\varepsilon(0)\times C_k^\perp(x)$, where $B_\varepsilon(0)$ is an $\varepsilon$-ball, the $\Phi$-invariant measure $\mu$ has a local product structure of $\mbox{Leb}\times\nu_{x,k}$, where $\mbox{Leb}$ is Lebesgue measure and $\nu_{x,k}$ is a measure on a local transversal $C_k^\perp(x)$. This measure assigns to each local transversal $C^\perp_k(x)$ the measure $\nu_{x,k}(C_k^\perp(x))$. Since $\{\nu_{x,k}\}$ forms a system of transverse invariant measures for the $\mathbb{R}^d$ action in the sense of \cite{BM:UE}, they will all be denoted by $\nu$ unless not doing so leads to ambiguity. Define
      $$\hat\nu_{x,k} := \nu_{x,k}(C_k^\perp(x)).$$
      \begin{lemma}
        \label{lem:measScale}
        There exists a $C_\mu>1$ such that for any $x\in\Omega$ and $k\in\mathbb{N}_0$,
        $$C_\mu^{-1}\lambda^{-k}\leq  \hat\nu_{x,k} \leq C_\mu\lambda^{-k}.$$
      \end{lemma}
      \begin{proof}
        Since $\Phi$ preserves $\mu$, $\mu$ has local product structure $\mathrm{Leb}\times \nu$, and the Lebesgue measure scales as $\lambda = \mathrm{det}(D_\gamma) = \prod \lambda_i$ under $\Phi$, $\nu $ scales as $\lambda^{-1}$ under $\Phi$. So it follows that $\nu(C_1^\perp(x)) = \nu(\Phi(C_0^\perp(\Phi^{-1}(x)))) = \lambda^{-1}\nu(C_0^\perp(\Phi^{-1}(x)))$. By compactness there are finitely many values of $\nu(C_0^\perp(\Phi^{-1}(x)))$, and so iterating the calculation gives the desired bound.
      \end{proof}
      \subsection{Function spaces}
      \label{subsec:functions}
      Let $C^r(\Gamma)$ be the space of functions on the branched manifold which are $C^r$ smooth at the branch set. More precisely, let $i_k:B_k\rightarrow \Gamma$ be the inclusion of the $k^{th}$-dimensional part of the branched set. A first order differential operator $X$ on $\Gamma$ is one for which there exists a first order differential operator $X_k$ on $B_k$ which can be extended to $X$ on $\Gamma$, that is, so that that $X_k i^*_k (f) = i_k^*(Xf)$ for all $k$. In other words, $C^r(\Gamma)$ is the set of functions $f\in C^r(\Gamma)$ such that
     \begin{equation}
        \label{eqn:CrCond}
        \begin{tikzcd}
          C^r(\Gamma) \arrow[r, "i_k^*"] \arrow[d, "X"']
          & C^r(B_k) \arrow[d, "X_k"] \\
          C^{r-1}(\Gamma) \arrow[r, "i_k^{*}"']
          & C^{r-1}(B_k)
        \end{tikzcd}
      \end{equation}
     for some $X_k$ and for all $k>0$. Thus a first order differential operator $X$ can be identified with a $d$-tuple $\{X_1,\dots, X_d\}$ of first order differential operators, where $X_k$ on $B_k$ satisfies (\ref{eqn:CrCond}).

     Another way to characterize $C^r(\Gamma)$ is as follows: since $\mathbb{R}^d$ acts on $\Omega$ by translation along unstable leaves, any $v\in\mathbb{R}^d$ defines the leafwise differential operators $\partial_{v}$ as
      \begin{equation}
        \label{eqn:leafwiseDer}
        \partial_{v}f(z) = \lim_{s\rightarrow 0^+}\frac{f(z+sv)- f(z)}{s}.
      \end{equation}
      Let $C^r(\Omega)$ be the set of leafwise-$C^r$ functions on $\Omega$ with respect to the family of operators $\partial_{e_1},\dots, \partial_{e_d}$. Thus $f\in C^r(\Gamma_k)$ if and only if $\pi^*_kf\in C^r(\Omega)$. Let
      $$C^r_{tlc}(\Omega) := \bigcup_{k}\pi_k^*C^r(\Gamma_k)\hspace{.75in}\mbox{ and }\hspace{.75in} C^\infty_{tlc}(\Omega) := \bigcup_{k}\pi_k^*C^\infty(\Gamma_k).$$
      For a $h\in C^\infty_{tlc}(\Omega)$ and fixed $p\in\Omega$, the function $f_h(t):=h\circ \varphi_t(p):\mathbb{R}^d\rightarrow \mathbb{R}$ is called \textbf{$p$-equivariant}.

      For a function $f$ on $\Omega$ define the transversal H\"older seminorm
      \begin{equation}
        \label{eqn:HoldSemi}
        |f|^\perp_\alpha = \sup_{x\in \Omega_\gamma} \sup_{x\neq y\in C^\perp_0(x)}\frac{\left| f(x)-f(y)\right|}{d(x,y)^\alpha},
      \end{equation}
      and let $H_\alpha^\perp(\Omega)$ be the Banach space of transversally $\alpha$-H\"older functions with norm
      \begin{equation}
        \label{eqn:TransvNorm}
        \|f\|^\perp_\alpha = \|f\|_{C^0} + |f|^\perp_\alpha.
      \end{equation}
      Note that for any $0<\beta<\alpha$ there is a compact inclusion $H^\perp_\alpha \subset H^\perp_\beta$.     
      \begin{remark}
        Since the local transversals $C^\perp_k$ are totally disconnected sets, the spaces $H_\alpha^\perp(\Omega)$ are non-trivial for \emph{every real} $\alpha>0$. This is a significant difference from stable sets which are smooth, and $\alpha$ here controls the analogous quantitity for smoothness in totally disconnected direction of $\Omega$.
      \end{remark}
      Let $C^r(\Omega)$ denote the set of functions $f$ on $\Omega$ such that $\partial_{v_{i_1}}\cdots\partial_{v_{i_r}} f$ is continuous for any choice of $r$ vectors $v_{i_j}$ in an orthonormal basis $\{v_1,\dots v_d\}$ of $\mathbb{R}^d$. Finally, define for $r\in\mathbb{N}$ and $\alpha>0$,
      \begin{equation}
        \label{eqn:Cr}
        C^r_\alpha(\Omega) = \left\{f\in C^r(\Omega)\cap H^\perp_\alpha(\Omega): \sum_{0\leq p \leq r}\sum_{|i|=p} \|\partial^if\|_{C^0}+|\partial^i f|^\perp_\alpha < \infty \right\}
      \end{equation}
      to be the space of functions which not only are $C^r$ smooth in the leaf direction but also whose derivatives up to order $r$ are transversally $\alpha$-H\"older. Here the traditional multiindex notation $i = (i_1,\dots, i_d)$ has been used, with $|i|=i_1+\cdots + i_d$. This is a Banach space under the norm
      \begin{equation}
        \label{eqn:CraNorm}
        \|f\|_{r,\alpha}:= \sum_{0\leq p \leq r}\sum_{|i|=p} \|\partial^if\|_{C^0}+|\partial^i f|^\perp_\alpha.
      \end{equation}

            Given a function $f\in L^1(\Omega)$ and $k\in \mathbb{N}_0:= \mathbb{N}\cup\{0\}$ set
      \begin{equation}
        \label{eqn:idem}
        \Pi_k f(x):= \hat\nu_{k,x}^{-1}\int_{C_k^\perp(x)} f(z)\, d\nu_{k,x},\;\;\;\; \Delta_kf := f-\Pi_k f,\;\;\;\mbox{ and }\;\;\;\delta_kf := \Pi_kf-\Pi_{k-1}f.
      \end{equation}
      These functions will be crucial for most results of this paper, so let me take some time to dicuss how one can think of them. Any $f\in L^1_\mu(\Omega)$ is a function of infinitely many variables, since this is how the solenoid is defined. The function $\Pi_k f$ is obtained by integrating along local transversals and consequently $\Pi_kf$ is transversally locally constant. This means that $\Pi_k f$ only depends on finitely many coordinates, that is, there is a function $g_k\in L^1(\Gamma)$ such that $\Pi_k f = \pi^*_k g_k$. Thus, $\Pi_kf$ can be thought of as the best approximation to $f$ if we can only consider the first $k$ coordinates. We will see below that $\Pi_kf\rightarrow f$ in a satisfying sense as $k\rightarrow \infty$. 

      If $\Pi_k f$ is an approximation to $f$ using the first $k$ coordinates, then $\Delta_k f$ is the error in this approximation. If $\Pi_kf\rightarrow f$ in some sense, then one should expect that $\Delta_kf\rightarrow 0$ in the same sense (this will come up later). Finally, if $\Pi_kf$ is an approximation to $f$ using the first $k$ coordinates, then $f_k:=\delta_k f$ is the best approximation to $f$ using \emph{only} the $k^{th}$ coordinate. So after setting $\Pi_{-1}f =0$, the approximation $\Pi_k f$ can be written as the finite sum
      $$\Pi_k f = \sum_{i=0}^k \delta_k f = \sum_{i=0}^k f_k$$
      and so if $\Pi_kf\rightarrow f$ then $f$ can be written as a sum $f = \sum f_k$ \textbf{in a canonical way}. These of course are rough descriptions of how one can think of these functions but since they will appear regularly in this paper one should might as well have a good way to think of them.

      A more precise description of what is happening involves conditional expectation. For $k\in\mathbb{N}_0$, let $\mathcal{A}_k$ be the $\sigma$-algebra generated by the preimages $\pi^{-1}_k(A)$ of Borel sets $A\subset \Gamma_k$. This is an increasing sequence of sub $\sigma$-algebras of the Borel $\sigma$-algebra $\mathcal{A}$ of $\Omega$. The conditional expectation $E(\cdot|\mathcal{A}_k):L^1(\Omega,\mathcal{A},\mu)\rightarrow L^1(\Omega,\mathcal{A}_k,\mu)$ map coincides with $\Pi_k$, that is, for any $f\in L^1(\Omega,\mu)$, $E(f|\mathcal{A}_k) = \Pi_kf$, and $\nu_{k,x}$ is the conditional measure of this conditional expectation. By the increasing martingale theorem \cite[Theorem 5.5]{EW:book}, $\Pi_kf\rightarrow f$ almost everywhere and in $L^1$.
      
      Using the notation $f_k = \delta_k f$ from above, for $r,\alpha \geq 0$, let
      $$\mathcal{S}^r_\alpha(\Omega):=\left\{f:\Omega\rightarrow \mathbb{R}:f = \sum_{k\geq 0}f_k,\substack{\displaystyle f_k= \pi^*_k(f^{(k)})\mbox{ for some $f^{(k)}\in C^r(\Gamma)$ and} \\ \displaystyle\mbox{ there exists a $C_f$ such that }\|f^{(k)}\|_{C^r(\Gamma)}\leq C_f \lambda_0^{-k\alpha}}.\right\}$$
      Let me make two comments which motivate the definition of these function spaces. The first one comes from the algebraic setting: if $S$ is the inverse limit of locally expanding affine linear maps of $\mathbb{T}^d$, then $L^2_\mu$ is spanned by a Fourier basis, and so one needs to quantify the decay rates of the Fourier coefficients to capture degrees of regularity. The generalization of this idea leads to the space $\mathcal{S}^r_\alpha$ as defined above. The second reason is that the representation of a function as a sum of pullbacks of distinct functions on approximants is canonical: if $f = \pi^*_0f^{(0)}$ then $f$ is transversally locally constant, and so $\delta_kf = 0$ for all $k>0$ and so the $f$ is uniquely represented as a sum of finitely many terms.
      \begin{proposition}
        \label{prop:spaceProp}
        For the spaces $\mathcal{S}^r_\alpha(\Omega)$ and $C^r_\alpha(\Omega)$ defined above with $r\in\mathbb{N},\alpha>0$:
        \begin{enumerate}
        \item for any $v\in \{v_1,\dots, v_d\}$, if $f\in \mathcal{S}^r_\alpha(\Omega)$ then $\partial_{v}f\in \mathcal{S}^{r-1}_{\alpha+1}(\Omega)$,
        \item for any $\varepsilon\in (0,\alpha)$, $\mathcal{S}^r_\alpha(\Omega)\subset C^r_{\alpha}(\Omega)$ densely with respect to the norm of $C^r_{\alpha-\varepsilon}(\Omega)$.
        \end{enumerate}
      \end{proposition}
      \begin{remark}
        I would like to remark on the surprising feature (i): taking leafwise derivatives \emph{increases} regularity in the transverse direction. Since this is the H\"older regularity, this depends on the metric used. The form adopted here is the one which corresponds to the natural choice of transversal metric in (\ref{eqn:metric}), as well as using the same number $\lambda_0$ in defining the spaces $\mathcal{S}^r_\alpha(\Omega)$. 
      \end{remark}
      \begin{proof}[Proof of Proposition \ref{prop:spaceProp}]
        For (i), if $f\in \mathcal{S}^r_\alpha$ then using (\ref{eqn:Rdaction}) and (\ref{eqn:leafwiseDer})
        $$\partial_v f(z) = \sum_{k\geq 0} \partial_v f_k(z) = \sum_{k\geq 0} \pi^*_k\left( \nabla f^{(k)} \cdot D^{-k}_\gamma v\right)(z) $$
        and
        $$\| \nabla f^{(k)}\cdot D^{-k}_\gamma v\|_{C^{r-1}(\Gamma_k)}\leq \lambda_0^{-k}\|f\|_{C^r(\Gamma_k)}\leq C_f\lambda_0^{-k}\lambda_0^{-\alpha k} = C_f \lambda_0^{-(\alpha+1) k}.$$
        So $\partial_v f \in \mathcal{S}_{\alpha+1}^{r-1}(\Omega)$.
        
        For (ii), for $f\in C^r_{\alpha}$, consider the approximations $\Pi_k f$. For every $k$, these approximations are all in $\mathcal{S}^r_\alpha$ for any $\alpha>0$ since they are transversally locally functions. What needs to be proved is that for any multiindex $i$ with $|i|\leq r$
        $$\|\partial^i\Pi_kf - \partial^if\|_{C^0}+ |\partial^i\Pi_kf - \partial^if|^\perp_{\alpha-\varepsilon}\rightarrow 0$$
        as $k\rightarrow \infty$.

        Let $i$ be one such multiindex. Then
        \begin{equation}
          \label{eqn:approx1}
          \begin{split}
            |\partial^i\Pi_kf(x) - \partial^if(x)| &= \hat\nu_{k,x}^{-1}\left|\int_{C^\perp_k(x)}\partial^if(z)-\partial^if(x)\, d\nu_{k,x} \right| \\
            &\leq \hat\nu_{k,x}^{-1}\int_{C^\perp_k(x)}\left|\partial^if(z)-\partial^if(x)\right|\, d\nu_{k,x} \leq C_{\partial^i f} \lambda_0^{-\alpha k}
          \end{split}
        \end{equation}
        and so $\|\partial^i\Pi_kf - \partial^if\|_{C^0}\rightarrow 0$. Here it was used that not only $f$ but its derivatives are transversally $\alpha$-H\"older.

        Let $x\in \Omega$ and suppose that $a,b\in C^\perp_\ell(x)$. Then if $\ell\leq k$:
        \begin{equation*}
            \frac{\left| (\Pi_k\partial^i f - \partial^if)(a)- (\Pi_k\partial^i f - \partial^if)(b)\right|}{\lambda_0^{-(\alpha-\varepsilon)\ell}} \leq \lambda_0^{(\alpha-\varepsilon)\ell}2C_{\partial^i f}\lambda_0^{-k\alpha}\leq 2C_{\partial^if}\lambda_0^{-\varepsilon k},
        \end{equation*}
        where (\ref{eqn:approx1}) was used in simplifying the numerator. If $\ell>k$:
        \begin{equation*}
            \frac{\left| (\Pi_k\partial^i f - \partial^if)(a)- (\Pi_k\partial^i f - \partial^if)(b)\right|}{\lambda_0^{-(\alpha-\varepsilon)\ell}} =  \lambda_0^{(\alpha-\varepsilon)\ell}|\partial^if(a)-\partial^if(b)|\leq C_{\partial^if}\lambda_0^{-\varepsilon k}.
        \end{equation*}
        And so $|\partial^i\Pi_kf - \partial^if|^\perp_{\alpha-\varepsilon}\rightarrow 0$, so the proof is concluded.
%
      \end{proof}

      \section{Transversal H\"older cohomology for Wieler solenoids}
      \label{sec:cohomology}
      Wieler solenoids, being defined as an inverse limit of surjective and locally-expaning maps have the advantage that their (\v Cech) cohomology can be explicitly computed. At the most basic level, this can be done with coefficients in $\mathbb{Z}$ through
      $$\check{H}^*(\Omega;\mathbb{Z}) = \lim_{n\rightarrow\infty} \left(\check{H}^*(\Gamma;\mathbb{Z}),\gamma^*\right).$$
      By the universal coefficient theorem and universality of inverse limits, the cohomology with real coefficients is
      \begin{equation}
        \label{eqn:checkR}
        \check{H}^*(\Omega;\mathbb{R}) = \lim_{n\rightarrow\infty} \left(\check H^*(\Gamma;\mathbb{R}),\gamma^*\right).
      \end{equation}

      Likewise, since there is a free $\mathbb{R}^d$ action on $\Omega$, there is an associated \textbf{Lie-algebra cohomology} defined as follows. Let $X_1,\dots, X_d$ be an orthonormal frame of $\mathbb{R}^d$. Define the operators $\partial_i$ using $X_i$ as in (\ref{eqn:leafwiseDer}). Let $C^\infty(\Omega;\Lambda\mathbb{R}^{d*})$ the space of leafwise smooth sections of $\Omega$ to $\Lambda\mathbb{R}^{d*}$, the graded exterior algebra of $\mathbb{R}^{d*}$. In other words $C^\infty(\Omega;\Lambda\mathbb{R}^{d*})$ is the space of functions $f:\Omega\rightarrow \Lambda\mathbb{R}^{d*}$ such that $\partial_{i_1}\cdots \partial_{i_k}f\in C^\infty(\Omega;\Lambda\mathbb{R}^{d*})$ for any finite collection of indices $i_1,\dots, i_k$.

      Let $d:C^\infty(\Omega;\Lambda^k\mathbb{R}^{d*})\rightarrow C^\infty(\Omega;\Lambda^{k+1}\mathbb{R}^{d*}) $ be the exterior differential operator defined as follows. Let $\{dx_1,\dots, dx_d\}$ denote the dual frame to $\{X_1,\dots, X_d\}$. Then
      $$d(f dx_{i_1}\wedge\cdots \wedge  dx_{i_k}) = \sum_{i=1}^d \partial_if dx_i\wedge dx_{i_1}\wedge \cdots \wedge dx_{i_k}.$$
      It is immediate to verify that this operator satisfies $d^2=0$. Moreover, there is natural subcomplex $C^\infty_{tlc}(\Omega;\Lambda^k\mathbb{R}^{d*})$ of sections with tlc coefficients.
      \begin{definition}
        The \textbf{Lie-algebra cohomology} $H^*(\Omega)$ of $\Omega$ is the cohomology of the differential complex $(C^\infty(\Omega;\Lambda\mathbb{R}^{d*}),d)$. That is,
        $$H^\bullet(\Omega):= \frac{\ker \left\{d:C^\infty(\Omega;\Lambda^\bullet\mathbb{R}^{d*})\rightarrow C^\infty(\Omega;\Lambda^{\bullet+1}\mathbb{R}^{d*})\right\}}{\mathrm{Im}\, \left\{d:C^\infty(\Omega;\Lambda^{\bullet-1}\mathbb{R}^{d*})\rightarrow C^\infty(\Omega;\Lambda^\bullet\mathbb{R}^{d*})\right\}}.$$
        The \textbf{transversally locally constant (tlc)} Lie-algebra cohomology is the cohomology of the subcomplex of tlc functions:
        $$H^\bullet_{tlc}(\Omega):= \frac{\ker \left\{d:C^\infty_{tlc}(\Omega;\Lambda^\bullet\mathbb{R}^{d*})\rightarrow C^\infty_{tlc}(\Omega;\Lambda^{\bullet+1}\mathbb{R}^{d*})\right\}}{\mathrm{Im}\, \left\{d:C^\infty_{tlc}(\Omega;\Lambda^{\bullet-1}\mathbb{R}^{d*})\rightarrow C^\infty_{tlc}(\Omega;\Lambda^\bullet\mathbb{R}^{d*})\right\}}.$$
      \end{definition}
      Finally, there is one more type of cohomology which is relevant here: since for $p\in \Omega$ a form $\eta\in C^\infty(\Omega;\Lambda^*\mathbb{R}^{d*})$ defines a $p$-equivariant form $\omega_\eta(t) := \varphi^*_t\eta:\mathbb{R}^d\rightarrow \Lambda^*\mathbb{R}^{d*}$, the \textbf{$p$-equivariant cohomology $H^*_p(\Omega;\mathbb{R})$} is defined as the cohomology of this subcomplex of the de Rham complex of $\mathbb{R}^d$.

      Kellendonk and Putnam \cite{KP:RS} proved that the $p$-equivariant cohomology is isomorphic to the tlc Lie-algebra cohomology, which in turn is isomorphic to the \v Cech cohomology (\ref{eqn:checkR}) with coefficients in $\mathbb{R}$. In general it is not true that $H^\bullet_{tlc}(\Omega)$ is isomorphic to $H^\bullet(\Omega)$. These two cohomologies capture two extremes of regularity in the transverse direction: functions in $C^\infty$ satisfy $|f|^\perp_0<\infty$ whereas $g\in C^\infty_{tlc}$ has $|g|_\alpha^\perp<\infty$ for all $\alpha\geq 0$. Thus it is natural to ask how much regularity is needed in the transverse direction to recover the real \v Cech cohomology. This is the transversally H\"older cohomology.
      
      The rest of the section is devoted to proving that instead of looking at the cohomology of $C^\infty_{tlc}$ sections, one could consider the cohomology of sections with coefficients in $\mathcal{S}^r_\alpha(\Omega)$ or $C^r_\alpha$. To do this, a name for this type of cohomology is needed.
      \begin{definition}
        The space $\Psi^m_{r,\alpha}$ is the space of $m$-forms $\eta:\Omega\rightarrow \Lambda^m\mathbb{R}^{d*}$ with coefficients in $\mathcal{S}^{r}_{\alpha}$. That is, an element $\Psi^m_{r,\alpha}$ can be written as
        $$\eta = \sum_{I\in I_m} \eta_I dx_I,$$
        where each $dx_I$ is a $m$-form of the form $dx_{i_1}\wedge\cdots \wedge dx_{i_m}$, and $\eta_I\in \mathcal{S}^{r}_{\alpha}$ for all $I\in I_m$. Let
        $$\mathcal{Z}^m_{r,\alpha}:= \ker \left\{d:\Psi^m_{r,\alpha}\rightarrow \Psi^{m+1}_{r-1,\alpha+1}\right\}\hspace{.5in}\mbox{ and }\hspace{.5in}\mathcal{B}^m_{r,\alpha}:= \mathrm{Im}\, \left\{d:\Psi^{m-1}_{r+1,\alpha-1}\rightarrow \Psi^{m}_{r,\alpha} \right\},$$
        and define the \textbf{$\mathcal{S}^{r}_{\alpha}$-cohomology of $\Omega$} as
        $$\mathcal{H}_{r,\alpha}^*(\Omega):= \mathcal{Z}^*_{r,\alpha}/\mathcal{B}^*_{r,\alpha},$$
        which is the Lie algebra cohomology with coefficients in $\Psi_{r,\alpha}^*$.

      \end{definition}
      \begin{theorem}
        \label{thm:mainCoh}
        Fix $r\in\mathbb{N}$ and $\alpha>1$ and let $\eta\in \mathcal{Z}^m_{r,\alpha}$. Then there exists $\eta'\in C^\infty_{tlc}(\Omega;\Lambda^m\mathbb{R}^{d*})$ and $\omega\in \Psi^{m-1}_{r+1,\alpha-1} $ such that $\eta-\eta' = d\omega$. That is, for $r\geq 1$ and $\alpha>1$:
        $$\mathcal{H}^*_{r,\alpha}(\Omega)\cong H^*_{tlc}(\Omega)\cong \check{H}^*(\Omega;\mathbb{R}). $$
      \end{theorem}
      The rest of this section is devoted to the proof of this theorem. First we start by reviewing de Rham cohomology for branched manifolds, then de Rham regularization for branched manifolds, and finally we put all of this together in the inverse limit structure.

      Let $H^*_{r,tlc}(\Omega)$ be the Lie-algebra cohomology with $C^r_{tlc}$ coefficients. That is, two forms $\eta_1,\eta_2$ are in the same cohomology class if there exists a $\omega\in C^r_{tlc}$ such that $\eta_1-\eta_2 = d\omega$.
      \begin{proposition}
        \label{prop:midCoh}
        For $r\in\mathbb{N}$, $H^*_{r,tlc}(\Omega)\cong H^*_{tlc}(\Omega)$.
      \end{proposition}
      \begin{proof}
        The goal is to show that for any closed $\eta\in C^r_{tlc}(\Omega;\Lambda^*\mathbb{R}^{d*})$ there is a $\eta'\in C^\infty_{tlc}(\Omega;\Lambda^*\mathbb{R}^{d*}) $ and $\omega\in C^{r}_{tlc}(\Omega;\Lambda^*\mathbb{R}^{d*})$ such that $\eta-\eta' = d\omega$. This will be done through the use of de Rham regularization \cite[\S III.15]{dR:book} applied to $p$-equivariant cohomology, since it is isomorphic to Lie-algebra cohomology.

        Let $p\in \Omega$ and let $\eta:\mathbb{R}^d\rightarrow \Lambda^*\mathbb{R}^{d*}$ be a $p$-equivariant form with $C^r$ coefficients. This means there is a $R_\eta>1$ such that if the tiling around the point $x$ in a ball of radius $R_\eta$ is the same as the one around $y$ in a ball of radius $R_\eta$ then $\eta(x) = \eta(y)$. This follows form the fact that $\eta(t) = \bar{\eta}\circ \varphi_t(p)$ for a transversally locally constant form $\bar{\eta}$. As such, there is a $k$ such that $\bar{\eta}$ is constant on $C_k^\perp(x)$ for any $x$, and so if $x,y\in\mathbb{R}^d$ are such that $\varphi_y(p) \in C_k^\perp(\varphi_x(p))$, then $\eta(x) = \eta(y)$.

        Pick $\varepsilon\in(0,1/2)$ and let $\mathcal{U} = \{U_i\}_{i\in\mathbb{N}}$ be a cover of $\mathbb{R}^d$ such that:
        \begin{enumerate}
        \item $\mathcal{U}$ is locally finite;
        \item $U_i$ is a Euclidean ball of radius $1+\varepsilon$ for all $i$;
        \item $\mathcal{U}$ is $p$-equivariant with radius $6R_\eta$, that is, if $U_i$ is centered at $x_i$ and $y$ has the same pattern as $x_i$ does inside a ball of radius $6R_\eta$ centered at $y$, then $y$ is the center of some $U_j\in\mathcal{U}$. 
        \end{enumerate}
        This type of cover is called a \textbf{$p$-equivariant cover adapted to $\eta$}. A way to construct this type of cover is to cover $\Gamma_k$ with finitely many balls of the appropriate size so that when they lift to $\Omega$, the intersection of leaves with these lifts are balls of size $1+\varepsilon$.
        
        Now de Rham regularization can be invoked \cite[Theorem 12 in \S III.15]{dR:book}: there exist operators $R$ and $A$ such that
        \begin{enumerate}
          \item $R\eta - \eta = dA\eta + Ad\eta$,
          \item $R\eta \in C^\infty(\mathbb{R}^d;\Lambda^*\mathbb{R}^{d*})$, and
          \item if $\eta \in C^r(\mathbb{R}^d;\Lambda^*\mathbb{R}^{d*})$, then $A\eta \in C^r(\mathbb{R}^d;\Lambda^*\mathbb{R}^{d*})$.
        \end{enumerate}
        Thus if $d\eta = 0$ then $dR\eta = 0$, and they differ by the exact $C^r$ form $dA\eta$. It remains to show that $A\eta$ is $p$-equivariant, that is, that is has something to do with $\Omega$.

        Since $\mathcal{U}$ is a $p$-equivariant cover adapted to $\eta$, if $x,y\in\mathbb{R}^d$ have neighborhoods of radius $6R_\eta$ which are translation equivalent, then $\eta\circ \varphi_\tau(x)=\eta\circ \varphi_\tau(y)$ for all $\tau$ such that $\|\tau\|\leq 5R_\eta$ and the union of the sets $U_1^x,\dots, U_k^x$ which contain $x$ is a set which is translation equivalent to the union of the sets $U_1^y,\dots U^y_k$ which contain $y$. As such the regularization at $x$ is the same as the one at $y$, that is, $A\eta(x) = A\eta(y)$, which means that $A\eta$ is $p$-equivariant.
      \end{proof}
      
      Let $\Gamma$ be a flat branched manifold of dimension $d$, where the set of branches is denoted by $B\subset \Gamma$. Denote by $I^k$ the $k$-skelleton of $\Gamma$. For each $x\in B_k:= B\cap (I^k\setminus I^{k-1})$ there is a natural tangent space $T_x\Gamma$ of dimension $k$ and a corresponding cotangent space. Let $\Delta_{\ell}^{k}$ be the set of smooth $\ell$-forms on $B_k$, on which there is the usual de Rham differential operator $d_k: \Delta_{\ell}^{k}\rightarrow \Delta_{\ell+1}^{k}$ and which are included into $\Gamma$ using maps $i_k:B_k \rightarrow \Gamma$. To give the entire branched space a smooth structure, consider the space of smooth maps
      $$\Delta_k(\Gamma)= \{\omega:\Gamma\rightarrow \Lambda^k\mathbb{R}^{d*}\}$$
      with coboundary operators $d:\Delta_k(\Gamma)\rightarrow \Delta_{k+1}(\Gamma)$ satisfying $d_k\circ i_k^* = i_k^*\circ d$ for all $k\geq 0$, that is, the operators which are conjugated to $d_k$ by $i_k^*$ for all $k$:
      \begin{equation}
        \label{eqn:flatCond}
        \begin{tikzcd}
          \Delta_\ell(\Gamma) \arrow[r, "i_k^*"] \arrow[d, "d"']
          & \Delta_\ell^k(B_k) \arrow[d, "d_k"] \\
          \Delta_{\ell+1}(\Gamma) \arrow[r, "i_k^{*}"']
          & \Delta_{\ell+1}^k(B_k)
        \end{tikzcd}
      \end{equation}
      Another way to characterize $\Delta_m(\Gamma)$ is the set of $m$-forms $\eta$ on $\Gamma$ such that $\pi^*_k\eta$ is a $C^\infty$ $m$-form on $\Omega$.
      \begin{definition}
        The \textbf{de Rham cohomology $H_{dR}^*(\Gamma)$ of the flat branched manifold $\Gamma$} is the cohomology of the complex $(\Delta_*(\Gamma),d)$ of smooth forms satisfying (\ref{eqn:flatCond}).
      \end{definition}
      Sadun \cite{sadun:deRham} proved that $H_{dR}^*(\Gamma)$ is isomorphic to the real \v Cech cohomology of $\Gamma$.
      \begin{lemma}
        \label{lem:normComp}
        For $\rho\in [1,\infty]$, let $\|\cdot \|':H^*_{\rho,tlc}(\Omega)\rightarrow \mathbb{R}$ be a norm. Then there exists a $K$ depending on the norm and on $\gamma$ such that
        $$\|[\pi^*_k \eta]\|'\leq K \|\eta\|_{C^r(\Gamma_k)}$$
        for any closed $\eta\in\Delta_*(\Gamma_k) $ which is $C^r$, $k\geq 0$, and $r\in [1,\rho]$.
      \end{lemma}
      \begin{proof}
        Since $H^i_{tlc}(\Omega)$ is finite dimensional, by Proposition \ref{prop:midCoh}, it will suffice it to prove for some norm. Let $\|\cdot\|$ be some norm on $H_i(\Gamma_k;\mathbb{R})$ and denote by $\|\cdot \|$ the dual norm on $H^i(\Gamma_k;\mathbb{R})$. That is, for a closed $i$-form $\eta\in \Delta_i(\Gamma_k)$:
        $$\|[\eta]\| = \sup_{0\neq [c] \in H_i(\Gamma_k;\mathbb{R})} \left|\frac{[\eta]([c])}{\|[c]\|}\right| = \sup_{0\neq [c] \in H_i(\Gamma_k;\mathbb{R})} \|[c]\|^{-1}\left|\int_c\eta\right|.$$
        The $i^{th}$ skeleton of $\Gamma_k$ is finitely generated and so there is a collection of representative cycles $c_1,\dots, c_m$ of a basis of $H_i(\Gamma_k;\mathbb{Z})$. Thus, for an integral class $[c] = \sum_{j=1}^m a_j(c) c_j\in H_i(\Gamma_k;\mathbb{Z})$ the absolute value of $\int_c\eta$ can be bounded as
        \begin{equation}
          \begin{split}
            \left|\int_c\eta\right|&= \left|\sum_{j=1}^m a_j(c) \int_{c_j}\eta\right|\leq \sum_{j=1}^m \left|a_j(c)\right| \left|\int_{c_j}\eta\right| \leq \|\eta\|_{C^r(\Gamma_k)}\sum_{j=1}^m \left|a_j(c)\right| \mathrm{Vol}_i(c_j) \\
            &\leq K'\left(\max_j\mathrm{Vol}_i(c_j)\right) \|[c]\|\|\eta\|_{C^r(\Gamma_k)} \leq K''\|[c]\|\|\eta\|_{C^r(\Gamma_k)},
          \end{split}
        \end{equation}
        where $K'$ comes from the equivalence of the $L^1$ norm and $\|\cdot\|$ in $H_i(\Gamma_k;\mathbb{R})$, and $\mathrm{Vol}_i(c_j)$ is the $i$-dimensional volume of the cycle $c_j$. Here it was used that $\int_{c_j}\eta$ can be bounded by the volume of the cycle times the $C^1$ norm of $\eta$. Thus it follows that
        $$ \|[\pi^*_k\eta]\| = \|[\eta]\| = \sup_{0\neq c \in H_i(\Gamma_k;\mathbb{R})} \|c\|^{-1}\left|\int_c\eta\right| \leq K \|\eta\|_{C^r(\Gamma_k)}$$
        for any $r\geq 1$.
      \end{proof}
      The following completes the proof of Theorem \ref{thm:mainCoh}.
      \begin{proposition}
        \label{prop:midCoh2}
        If $r\in\mathbb{N},\alpha>1$ then $\mathcal{H}^*_{r,\alpha}(\Omega)\cong H^*_{tlc}(\Omega)$.
      \end{proposition}
      \begin{proof}
        Let $\eta:\Omega\rightarrow \Lambda^*\mathbb{R}^{d*}$ be a closed form with coefficients in $\mathcal{S}^r_\alpha(\Omega)$ with $r\geq 1$ and $\alpha>1$. By Proposition \ref{prop:midCoh}, it suffices to show that there is a form $\eta'\in H^*_{r,tlc}(\Omega)$ such that $\eta-\eta' = d\omega$ for some $\omega$ with coefficients in $\mathcal{S}^{r+1}_{\alpha-1}$. 

        By the definition of $\mathcal{S}^r_\alpha(\Omega)$ for each $k\geq 0$ there exists $\eta_k\in \Delta_*(\Gamma_k)$ such that $\eta$ is canonically expressed as $\eta = \displaystyle\sum_{k\geq 0} \delta_k\eta_k= \displaystyle\sum_{k\geq 0} \pi^*_k\eta_k$ (see \S \ref{subsec:functions}), where there are infinitely many non-zero terms, as that would otherwise make $\eta$ a tlc form. Note that this expression for $\eta$ has the property that $\eta_k\neq \gamma^* \eta'$ for some $\eta'\in \Delta_*(\Gamma_{k-1})$. Now define the sequence of forms
        $$\eta^{(n)}:= \sum_{k=0}^n \pi^*_k\eta_k.$$
        Each of these forms is a closed tlc form. Indeed, since
        $$0=d\eta = d\left(\sum_{k\geq 0} \pi^*_k\eta_k \right)= \sum_{k\geq 0} d\pi^*_k\eta_k =  \sum_{k\geq 0} d\delta_k \eta = \sum_{k\geq 0} \delta_k d\eta,$$
        since $d\Pi_k\eta = \Pi_k d\eta$ by the preservation of the transverse measure. Thus $d\delta_k\eta = \delta_kd\eta = \delta_k 0 = 0$ for all $k$.  It follows that $\eta^{(n)}$ has a tlc cohomology class $[\eta^{(n)}]\in H^*_{r,tlc}(\Omega)$. Observe that $[\eta^{(n)}]$ is a convergent sequence. Indeed, by Lemma \ref{lem:normComp}, for $n>m>0$:
        $$\left\|[\eta^{(n)}] - [\eta^{(m)}]\right\| \leq \left\|\sum_{k> m}[\pi^*_k\eta_k]\right\|\leq    K\sum_{k\geq m}\|\eta_k\|_{C^r(\Gamma_k)}\leq K'C_\eta\lambda_0^{-\alpha m}.$$
        Thus $[\eta]$ should be assigned the cohomology class $\lim_{n\rightarrow \infty}[\eta^{(n)}] \in H^*_{r,tlc}(\Omega)$
        and so the goal is to find a tlc representative of this class and show that $\eta$ is cohomologous to it.
        
        Recall the eventual range
        $$ER_i(\Omega):= (\gamma^*)^{\beta_i}\check{H}^i(\Gamma;\mathbb{R})\subset \check H^i(\Gamma;\mathbb{R}),$$
        where $\beta_i = \dim H^i(\Gamma;\mathbb{R})$. Any class $c\in \check H^i(\Omega;\mathbb{R})$ is represented by a class in $ER_i(\Omega)\subset \check H^i(\Gamma_{\beta_i};\mathbb{R})$. Since $\check H^i(\Gamma;\mathbb{R})$ is isomorphic to $H^i_{dR}(\Gamma)$ \cite[Appendix A]{sadun:deRham}, any class $c\in H^i_{tlc}(\Omega)$ has a representative $ \pi_{\beta_i}^* \eta_c$ coming from the $\beta_i^{th}$ projection map.

        Let $\pi^*_{\beta_i}\eta'_k$ be the form cohomologous to $\pi_k^*\eta_k$ and let
        $$\eta_{(n)} = \sum_{k=0}^n \pi^*_{\beta_i}\eta'_k = \pi^*_{\beta_i}\left(\sum_{k=0}^n \eta'_k\right).  $$
        Then $\eta_{(n)} - \eta^{(n)} = d\omega_n$ for all $n$. In addition,
        $$\eta_{(\infty)} := \lim_{n\rightarrow \infty}\sum_{k=0}^n \pi^*_{\beta_i}\eta'_k = \pi^*_{\beta_i}\left(\lim_{n\rightarrow \infty}\sum_{k=0}^n \eta'_k\right)$$
        is a $C^r$ tlc function which is cohomologous to $\eta$.
      \end{proof}
      
        
      \subsection{Application: Deviations of ergodic averages}
      \label{subsec:deviations}
      The spectrum of $\Phi^*:\check H^d(\Omega;\mathbb{R})\rightarrow \check H^d(\Omega;\mathbb{R})$ gives rates of convergence of ergodic integrals. This was first proved in \cite{sadun:exact} in the self-similar case and later in \cite{ST:SA} in the self-affine case. The class of functions used in those results were $C^\infty_{tlc}$. Theorem \ref{thm:mainCoh} implies that the same rates of convergence can now be given for functions in $\mathcal{S}^r_\alpha$ for $r\in\mathbb{N}$ and $\alpha>1$. The statement will be provided here without proof as Theorem \ref{thm:mainCoh} allows the argument in \cite{ST:SA} to carry over verbatim.

      Before stating the theorem, some notation needs to be established. Denote by $|\nu_1| > \dots> |\nu_r| > 0$ the norms of the $r$ distinct eigenvalues of the map $\Phi^*$ acting on $H^d$. Let $E_i$ be the generalized eigenspaces for the action of $\Phi$ on $H^d(\Omega;\mathbb{R})$ induced by the map $\Phi^*$ corresponding to the eigenvalue $\nu_i$. The subspaces $E_i$ are decomposed as
      $$E_i = \bigoplus_{j=1}^{\kappa(i)} E_{i,j},$$
      where $\kappa(i)$ is the size of the largest Jordan block associated with $\nu_i$, as follows.
      For each $i$, we choose a basis of classes $\{[\eta_{i,j,k}]\}$ with the property that $\langle [\eta_{i,j,1}],[\eta_{i,j,2}],\dots, [\eta_{i,j,s(i,j)}]\rangle = E_{i,j}$ and
      \begin{equation}
        \label{eqn:basisAct}
        \Phi^* [\eta_{i,j,k}] = \left\{\begin{array}{ll}
        \nu_i [\eta_{i,j,k}] + [\eta_{i,j-1,k}]   &\mbox{ for $j>1$,} \\
        \nu_i [\eta_{i,j,k}]  &\mbox{ for $j=1$.}
        \end{array}\right.
      \end{equation}
      \begin{definition}
        The \emph{rapidly expanding subspace} $E^+(\Omega) \subset H^d(\Omega)$ is the direct sum of all generalized eigenspaces $E_i$ of $\Phi^*$ such that the corresponding eigenvalues $\nu_i$ of $\Phi^*$ satisfy
        \begin{equation}
          \label{eqn:RES}
          |\nu_i|\geq \frac{\lambda}{\lambda_0}.
        \end{equation}
        The subspace $E^{++}\subset E^+$ consists of all vectors for which the inequality (\ref{eqn:RES}) is strict.
      \end{definition}      
      We order the indices of distinct subspaces of $E^+(\Omega)$ as follows. First, we set $I^+ = I^{+,>} \cup I^{+,=}$ be the index set of classes $[\eta_{i,j,k}]$ which form a generalized eigenbasis for $E^+(\Omega)$, where the indices in $I^{+,>}$ contains vectors corresponding to a strict inequality in (\ref{eqn:RES}) and the indices in $I^{+,=}$ correspond to vectors associated to eigenvalues which give an equality in (\ref{eqn:RES}). The set $I^{+,=}$ can be empty but $I^{+,>}$ always has at least one element.
      The set $I^+$ is partially ordered: $(i,j,k)\leq (i',j',k')$ if $L(i,j,T)T^{ds_i}\geq L(i',j',T)T^{ds_{i'}}$ for $T>1$, where
      \begin{equation}
        \label{eqn:logs}
        L(i,j,T) = \left\{\begin{array}{ll}
        (\log T)^{j-1} &\mbox{ if $\nu_i$ satisfies (\ref{eqn:RES}) strictly} \\
        (\log T)^{j} &\mbox{ if $\nu_i$ satisfies equality in (\ref{eqn:RES})}
        \end{array}\right.
      \end{equation}
      and $s_i = \frac{\log |\nu_i|}{\log \nu_1}$. The order does not depend on the indices $k$.
      
      By passing to a power we can assume that $D_\gamma\in GL^+(d,\mathbb{R}) = \mathrm{exp}( \mathfrak{gl}(d,\mathbb{R}))$. Let $a\in \mathfrak{gl}(d,\mathbb{R})$ be the matrix which satisfies $\exp (a) =D_\gamma$ and let $g_t = \exp(at)$. Letting $B_1$ denote the unit ball, define the averaging family $\{B_T\}_{T\geq 1}$ by
      \begin{equation}
        \label{eqn:rescalledSets}
        B_T =  g_{\sigma \log T} B_1,
      \end{equation}
      where $\sigma = d/\log\det A$. As such, we have that $\mathrm{Vol}(B_T)  = \mathrm{Vol}(B_1) T^d$. Let $\rho = \mathrm{dim}\, E^+(\Omega_\Lambda)$.
      \begin{theorem}
        \label{thm:deviations}
        For $r\in\mathbb{N}$ and $\alpha>1$, there exist a constant $C_{\gamma}$ and $\rho$ $\mathbb{R}^d$-invariant distributions $\{\mathcal{D}_{i,j,k}\}_{(i,j,k)\in I^+}$ such that, for any $f \in \mathcal{S}^r_\alpha(\Omega)$, if there is an index $(i,j,k)$ such that $\mathcal{D}_{i',j',k'} (f) = 0$ for all $(i',j',k')<(i,j,k)$ but $\mathcal{D}_{i,j,k} (f) \neq 0$, then for $T>3$ and any $x \in \Omega_{\gamma}$,
        $$\left| \int_{B_T} f\circ \varphi_s(x)\, ds \right| \leq C_{\gamma,f}  L(i,j,T)T^{d\frac{\log|\nu_{i}|}{\log \lambda}  }.$$
        Moreover, if $\mathcal{D}_{i,j,k}(f) = 0$ for all $(i,j,k)\in I^+_\Lambda$ then
        \begin{equation}
          \label{eqn:bdryErr2}
          \left| \int_{B_T} f\circ \varphi_s(\Lambda)\, ds \right| \leq C_f T^{d\left(1-\frac{\log|\lambda_0|}{\log\lambda}\right)}
        \end{equation}
        for all $T>1$. 
      \end{theorem}
      \section{Ruelle spectrum and quantitative mixing}
      \label{sec:ruelle}
      An application of transverser H\"older cohomology is the construction of anisotropic Banach spaces for solenoids. This section is dedicated to proving the quantitative mixing results for $\Phi$. To do this, it is necessary to introduce so-called anisotropic spaces of functions. I am particularly inspired by \cite{FGL:anosov, BKL:locating} and the PhD dissertation of D. Galli, and so I will follow some of the ideas there.
      \subsection{Anisotropic Banach spaces for flat Wieler solenoids}
      \label{subsec:aniso}
      Let $\varphi:\Omega\rightarrow \wedge^kT^*\mathbb{R}^d$ be a $m$-form. Analogous to (\ref{eqn:HoldSemi}), let
      \begin{equation}
        \label{eqn:HoldSemiMulti}
        |\varphi|^\perp_{\alpha,m} = \sup_{x\in \Omega_\gamma} \sup_{x\neq y\in C^\perp_0(x)}\frac{\left\| \varphi(x)-\varphi(y)\right\|}{d(x,y)^\alpha}.
      \end{equation}
      At first it may seem like $\varphi(x)-\varphi(y)$ is not defined, as each summand lives on a different fiber of the bundle. However, since the leaves of the foliation are flat and dense in $\Omega$, parallel transport makes this operation unambiguous. Define the space $H_{\alpha,m}^\perp$ of $\alpha$-H\"older $m$-forms to be those for which
      \begin{equation}
        \label{eqn:TransvNormMulti}
        \|\varphi\|^\perp_{\alpha,m} := \|\varphi\|_{C^0}+|\varphi|^\perp_{\alpha,m}<\infty.
      \end{equation}
      This is a Banach space, and if $\alpha>\beta$ we have that $H^\perp_{\alpha,m}\subset H^\perp_{\beta,m}$ compactly. Denote by $B_{\varepsilon,m}^\perp(\alpha)\subset H_{\alpha,m}^\perp$ the $\varepsilon$-ball with respect to (\ref{eqn:TransvNormMulti}). Note that $B_{1,m}^\perp(\alpha)\subset B_{1,m}^\perp(\alpha-\delta)$ for all $\delta$ small enough.

      For the sake of convenience it will be assumed that $D_\gamma$ has no Jordan blocks. As such, let $v_1,\dots, v_d$ be a normalized basis of $\mathbb{R}^d$ which are also eigenvectors for $D_\gamma$: $D_\gamma v_i = \lambda_i v_i$ with $\lambda_i>1$ and $\lambda = \lambda_1\cdots \lambda_d$. Given this choice, for any $i\in\{1,\dots, d\}$, $\partial_i$ will denote the differential operator $\partial_{v_i}$ and for a multiindex $i = (i_1,\dots, i_d)$ we will denote by $\partial^i= \partial^{i_1}_{1}\cdots\partial^{i_d}_d$ and $|i| = i_1+\cdots +i_d$. With this notation, let
      $$\vertiii{\eta}_{r,\alpha,m}:= \sum_{0\leq p \leq r}\sup_{|i|=p} \sup_{\varphi\in B_{1,m}^\perp(\alpha+p)} \left| \int_{\Omega} \left\langle\varphi, \partial^i \eta \right\rangle\, d\mu\right|,$$
      and let $\mathcal{B}^{r,\alpha}_{m}$ be the completion of $\Psi^m_{r,\alpha}$ with respect to $\vertiii{\cdot}_{r,\alpha,m}$ and set $\Lambda = \frac{\log\lambda}{\log\lambda_0}$.

      To the uninitiated reader, it is worth pointing out that functions in the anisotropic Banach spaces $\mathcal{B}_m^{r,\alpha}$ play two simultaneous roles, which become evident from the way the norm was defined: first, they serve as functions in the Euclidean variable, whereas in the transversal variable they serve the roles of currents. 

      Recall the Hodge-$\star$ operator which sends $m$-forms to $(d-m)$ forms $\star: \Psi^m_{r,\alpha}\rightarrow  \Psi^{d-m}_{r,\alpha}$. The invariant probability measure $\mu$ gives a canonical choice of the Lebesgue volume element on the $\mathbb{R}^d$-leaves of $\Omega$, and thus there is a canonical volume element $dt = \star 1$. With this choice, the Hodge-$\star$ operator gives a canonical bijection between functions and tangential $d$-forms: for $h\in \mathcal{S}^r_\alpha$, $\star h  = h(\star 1) = h\, dt\in \Psi^d_{r,\alpha}$.
      \begin{proposition}
        \label{prop:compact}
        $\mathcal{B}^{r',\alpha'}_{m}\subset \mathcal{B}^{r,\alpha}_{m}$ if $r'\geq r$ and $\alpha'\leq \alpha$, and the inclusion is compact if $r'>r$ and $\alpha>\alpha'+\Lambda>\Lambda>0$.
      \end{proposition}
        The inclusions follow from the definitions of the norms so what is left to prove is the compactness of the inclusion. 
        The following compactness criterion \cite[\S 2.2]{FGL:anosov} will be used: Let $\mathcal{B}\subset\mathcal{C}$ be two Banach spaces and assume that for any $\varepsilon>0$ there exist finitely many continuous linear forms $L_1,\dots, L_m$ on $\mathcal{B}$ such that for any $x\in\mathcal{B}$
        $$\|x\|_\mathcal{C}\leq \varepsilon \|x\|_\mathcal{B}+ \sum_{i\leq m}|L_i(x)|.$$
        Then the inclusion of $\mathcal{B}$ in $\mathcal{C}$ is compact.

        In order to apply the criterion, several estimates will need to be obtained. For an $m$-form $\varphi:\Omega\rightarrow\wedge^m T^*\mathbb{R}^d$ and $k\in\mathbb{N}_0$, set $\hat\nu_{k,x}:=\nu_{k,x}(C^\perp_k(x))$, recall (\ref{eqn:idem})
        \begin{equation}
          \label{eqn:idempotent}
          \Pi_k\varphi(x):= \hat\nu_{k,x}^{-1}\int_{C_k^\perp(x)} \varphi\, d\nu_{k,x},\;\;\;\;\;\mbox{ and }\;\;\;\;\;\Delta_k:= \varphi-\Pi_k\varphi,
        \end{equation}
        both of which are $m$-forms. Recall that $\Pi_k\varphi$ is transversally locally constant. That is, $\Pi_k\varphi(z)$ depends only on the first $k$ coordinates $z_i$ of $z$.
        \begin{lemma}
          \label{lem:a'bound}
          For $r\in\mathbb{N}_0$, let $\varphi:\Omega\rightarrow \wedge^m T^*\mathbb{R}^d$ with $\|\varphi\|^\perp_{\alpha+r,m}\leq 1$. If $\alpha>\Lambda+\alpha'>0$, then for any $\varepsilon>0$ there exists a $k>0$ such that
          $$\|\Delta_k\varphi\|^\perp_{\alpha'+r,m}<\frac{\varepsilon}{2}\hspace{.5in}\mbox{ and }\hspace{.5in}\|\Pi_k\varphi\|^\perp_{\alpha+r,m}\leq 4\lambda_0^\Lambda C_\mu\left(\frac{8C_\mu}{\varepsilon}\right)^{\frac{\Lambda}{r+\alpha-\Lambda-\alpha'}}.$$
        \end{lemma}
        \begin{proof}
        Using Lemma \ref{lem:measScale},
        $$\|\Delta_k\varphi(x)\|\leq \hat\nu_{k,x}^{-1}\int_{C_k^\perp(x)}\left\|\varphi(x)-\varphi(z)\right\|\, d\nu_{k,x}(z)\leq  C_\varphi\lambda_0^{-k(\alpha+r)} $$
        and so $\|\Delta_k\varphi\|_{C^0}\leq 2 C_\varphi\lambda_0^{-\left(\alpha+r\right)k}$.
        
        Now $\|\Delta_k\varphi\|_{\alpha',m}^\perp$ will be bound. If $a,b\in C^\perp_\ell(x)$ with $\ell\geq k$ then $C^\perp_k(a) = C_k^\perp(b)$ and so
        \begin{equation*}
          \begin{split}
            \|\Delta_k\varphi(a)-\Delta_k\varphi(b)\| &= \hat\nu_{k,a}^{-1}\left\|\varphi(a)-\varphi(b)\right\|\\
            &\leq C_\mu\lambda^k\lambda_0^{-\ell(\alpha+r)} = C_\mu\lambda_0^{k\Lambda-\ell(\alpha+r)}.
          \end{split}
        \end{equation*}
        If $\ell<k$ and $a,b\in C^\perp_\ell(x)$:
        \begin{equation*}
          \begin{split}
            \|\Delta_k\varphi(a)-\Delta_k(b)\|&\leq \hat{\nu}_{k,a}^{-1}\int_{C^\perp_k(a)}\|\varphi(a)-\varphi(z)\|\, d\nu_{k,a}(z) + \hat{\nu}_{k,b}^{-1}\int_{C^\perp_k(b)}\|\varphi(b)-\varphi(z)\|\, d\nu_{k,b}(z) \\
            &\leq 2C_\varphi\lambda_0^{-k(r+\alpha)}.
          \end{split}
        \end{equation*}
        Thus if $a,b\in C_\ell^\perp(x)$
        \begin{equation}
          \label{eqn:largeK}
          \frac{\|\Delta_k\varphi(a)-\Delta_k(b)\|}{\lambda_0^{-(\alpha'+r)\ell}}\leq \left\{\begin{array}{ll} C_\mu \lambda_0^{(\alpha'+r)\ell} \lambda_0^{\left(k\Lambda-\ell(\alpha+r)\right)}\leq C_\mu\lambda_0^{-k(\alpha-\Lambda-\alpha')}&\mbox{ if }\ell\geq k, \\ 2C_\mu \lambda_0^{(\alpha'+r)\ell}\lambda_0^{-\left(\alpha+r\right)k}\leq 2C_\mu \lambda_0^{-k(\alpha-\Lambda-\alpha')} & \mbox{ if }\ell<k. \end{array}\right.
        \end{equation}
        Putting everything together:
        \begin{equation*}
          \begin{split}
            \|\Delta_k\varphi\|^\perp_{\alpha'+r,m} &= \|\Delta_k\varphi\|_{C^0}+|\Delta_k\varphi|^\perp_{\alpha',m}\leq 2C_\mu\lambda_0^{(\Lambda-\alpha)k}+2C_\mu\lambda_0^{(\Lambda+\alpha'-\alpha)k} \\
            &\leq 4C_\mu\lambda_0^{-(\alpha-\Lambda-\alpha')k}.
          \end{split}          
        \end{equation*}
        Thus if
        \begin{equation}
          \label{eqn:goodK}
          k = \left\lceil \frac{\log\left(\frac{8C_\mu}{\varepsilon}\right)}{(\alpha-\Lambda-\alpha')\log\lambda_0} \right\rceil
        \end{equation}
        then $\|\Delta_k\varphi\|^\perp_{\alpha'+r,m}\leq \varepsilon/2$, which proves the first estimate.

        To obtain the second estimate, first note that for $a,b\in C^\perp_\ell(x)$, $\|\Pi_k\varphi(a)-\Pi_k\varphi(b)\| = 0$ if $\ell\geq k$ and if $\ell<k$:
        $$\|\Pi_k\varphi(a)-\Pi_k\varphi(b)\|\leq \|\varphi(a)-\varphi(b)\| + \|\Delta_k\varphi(a)-\Delta_k\varphi(b)\|\leq \lambda_0^{-\ell(r+\alpha)} + 2C_\mu \lambda_0^{(\Lambda-\alpha-r)k}, $$
        where the first estimate follows from the fact that $\|\varphi\|^\perp_{r+\alpha,m}\leq 1$ and the second from the estimate leading to (\ref{eqn:largeK}). Thus
        $$|\Pi_k\varphi|^\perp_{\alpha+r,m}\leq 1+2C_\mu \lambda_0^{k(\alpha+r)}\lambda_0^{(\Lambda-\alpha-r)k} = 1+2C_\mu \lambda_0^{\Lambda k}$$
        and so it follows that
        \begin{equation}
          \label{eqn:secondEst}
          \|\Pi_k\varphi\|^\perp_{\alpha+r,m} = \|\Pi_k\varphi\|_{C^0}+|\Pi_k\varphi|^\perp_{\alpha+r,m}\leq \|\varphi\|_{C^0} + 1+2C_\mu \lambda_0^{\Lambda k} \leq 2+2C_\mu \lambda_0^{\Lambda k}\leq 4C_\mu \lambda_0^{\Lambda k}.
        \end{equation}
        Using in (\ref{eqn:secondEst}) the choice for $k$ in (\ref{eqn:goodK}), the second estimate follows.
        \end{proof}
        \begin{proof}[Proof of Proposition \ref{prop:compact}]
          To apply the compactness criterion, let $\varepsilon>0$. Since $H^\perp_{\alpha+r,m}$ is compactly embedded in $H^\perp_{\alpha'+r,m}$ let $\{\varphi_j\}_{j\leq K_\varepsilon}\subset H^\perp_{\alpha+r,m}$ be a finite family of forms such that for any $\varphi\in H^\perp_{\alpha+r,m}$ with $\|\varphi_j\|^\perp_{\alpha+r,m}\leq 4\lambda_0^\Lambda C_\mu\left(\frac{8C_\mu}{\varepsilon}\right)^{\frac{\Lambda}{r+\alpha-\Lambda-\alpha'}}$ there is a $\varphi_{j^*}$ in the family such that $\|\varphi-\varphi_{j^*}\|_{\alpha'+r,m}^\perp\leq \varepsilon/2$. Define the finite family of linear forms $L_{j,i}:\mathcal{B}^{r',\alpha'}_m\rightarrow \mathbb{R}$ by
          $$L_{j,i}(\eta) = \int_{\Omega}\langle\varphi_j,\partial^i\eta\rangle\, d\mu,$$
          where $i$ is a multiindex of length at most $r'$ and $j\leq K_\varepsilon$.

          For $\varphi$ with $\|\varphi\|^\perp_{\alpha+r,m}\leq 1$, let $k\in\mathbb{N}$ be the one given by Lemma \ref{lem:a'bound}, so that $\|\Delta_k\varphi\|^\perp_{\alpha'+r,m}\leq \varepsilon/2$. Let $j^*\leq K_\varepsilon$ be such that $\|\Pi_k\varphi-\varphi_{j^*}\|^\perp_{\alpha'+r,m}\leq \varepsilon/2$. Thus for $\eta\in\mathcal{B}^{r',\alpha'}_m$ and a multiindex $i$ with $|i|=r$ it follows that
          $$\left|\int_\Omega \langle \varphi, \partial^i \eta\rangle\, d\mu\right|\leq \left|\int_\Omega\langle \Delta_k\varphi ,\partial^i\eta\rangle\, d\mu\right| + \left|\int_\Omega \langle\Pi_k\varphi-\varphi_{j^*},\partial^i\eta\rangle\, d\mu\right| + \left|\int_\Omega \langle \varphi_{j^*},\partial^i\eta\rangle\, d\mu\right| $$
          and so it follows that
          $$\vertiii{\eta}_{r,\alpha}\leq \frac{\varepsilon}{2}\vertiii{\eta}_{r',\alpha'}+ \frac{\varepsilon}{2}\vertiii{\eta}_{r',\alpha'} + \left|L_{j^*,i}(\eta)\right| $$
          and so compactness follows from the criterion.
      \end{proof}
        
        It will be useful below to have a version of cohomology with coefficients in the anisotropic spaces $\mathcal{B}_*^{r,\alpha}$. To that end, let $H_{\mathcal{B}^{r,\alpha}}^*(\Omega)$ be the cohomology of tangentially smooth forms with $\mathcal{B}^{r,\alpha}_0$ coefficients.
        \begin{proposition}
          \label{prop:AniCoh}
          Let $r\in\mathbb{N}$ and $\alpha>1$. Then $H_{\mathcal{B}^{r,\alpha}}^*(\Omega)\cong \mathcal{H}^*_{r,\alpha}(\Omega ) \cong \check H(\Omega;\mathbb{R})$.
        \end{proposition}
        \begin{proof}
          Let $\eta\in \mathcal{B}_m^{r,\alpha}$ be a closed $m$-form. It will be shown that there is a form $\eta'\in \Psi^m_{r,\alpha}$ such that $\eta-\eta' = d\omega$.

          Let $\{\eta_k\}\subset \Psi^m_{r,\alpha}$ be a sequence of closed forms such that $\eta_k\rightarrow \eta$ in $\mathcal{B}_m^{r,\alpha}$. As in the proof of Proposition \ref{prop:midCoh2}, for each $k$ there is a $\eta'_k\in \pi^*_{\beta_m} C^r(\Gamma)$ such that $\eta_k - \eta_k' = d\omega_k$. Thus
          $$\int_{\Omega}\langle\varphi,\eta-\eta_k\rangle\, d\mu = \int_{\Omega}\langle\varphi,\eta-\eta_k'-d\omega_k\rangle\, d\mu\rightarrow 0$$
          for any $\varphi\in H^\perp_\alpha(\Omega)$. Thus $[\eta-\eta'_k] = [\eta]-[\eta'_k]\rightarrow 0$. But $\eta'_k\in \pi^*_{\beta_m}C^r$ for all $k$, and so $[\eta'_k]\in \pi^*_{\beta_m}\check H^m(\Gamma;\mathbb{R})\subset \check H^m(\Omega;\mathbb{R})$ for all $k$, meaning that $[\eta]$ defines a class in $\check H^m(\Omega;\mathbb{R})$.
        \end{proof}
        Proposition \ref{prop:spaceProp} and the definition of the spaces $\mathcal{B}_m^{r,\alpha}$ implies the following.
        \begin{lemma}
          \label{lem:bounded}
          The differential operator $d$ on $m$ forms is a bounded operator from $\mathcal{B}^{r,\alpha}_m$ to $\mathcal{B}^{r-1,\alpha+1}_{m+1}$.
        \end{lemma}
        
        \subsection{The transfer operator}
        Define the transfer operator to be
        $$\mathcal{L}f := f\circ \Phi^{-1}$$
        where $f$ is an $m$-form and note that for any $y\in C_0^\perp(x)$:
        $$\left\|\Phi^*f(x)-\Phi^*f(y)\right\| \leq |f|^\perp_{\alpha,m} d(\Phi(x),\Phi(y))^\alpha\leq |f|^\perp_{\alpha,m}  \lambda_0^{-\alpha}d(x,y)^\alpha$$
        and so
        \begin{equation}
          \label{eqn:stableContr}
          |\Phi^*f|^\perp_{\alpha,m}\leq \lambda_0^{-\alpha} |f |^\perp_{\alpha,m}.
        \end{equation}
        The following estimates, similar to those of Lemma \ref{lem:a'bound}, will be needed in the proof for the Lasota-Yorke inequalities below.
        \begin{lemma}
          \label{lem:repeat}
          If $\|\varphi\|^\perp_{\beta,m}\leq 1$ then for $\alpha>\beta$:
          \begin{equation}
            \begin{split}
              \left \|(\Delta_k\varphi)\circ \Phi^n\right\|^\perp_{\beta,m} &\leq 2C_\mu \left(\lambda_0^{-\beta k}+\lambda_0^{\Lambda k -\beta n}\right)\hspace{.25in}\mbox{ and }\hspace{.25in}\\
              \|\Pi_k\varphi\circ \Phi^n\|^\perp_{\alpha,m} &\leq 1+3 C_\mu \lambda_0^{(\alpha-\beta)k-n\alpha}.
            \end{split}
          \end{equation}
        \end{lemma}
        \begin{proof}
          The estimate for $\|\Delta_k\varphi\|^\perp_{\beta,m}$ follows essentially from the computations up to and including (\ref{eqn:largeK}), and combined with (\ref{eqn:stableContr}) the estimate for $\left \|(\Delta_k\varphi)\circ \Phi^n\right\|^\perp_{\beta,m}$ follows. Now, if $a,b\in C^\perp_\ell(x)$ and $\ell\geq k$, then $\|\Pi_k\varphi(a)-\Pi_k\varphi(b)\|=0$. If $\ell<k$:
          $$\|\Pi_k\varphi(a)-\Pi_k\varphi(b)\|\leq\|\varphi(a)-\varphi(b)\|+\|\Delta_k\varphi(a)-\Delta_k\varphi(b)\|\leq \lambda_0^{-\ell\beta}+2C_\mu \lambda_0^{-\beta k},$$
          where the last term is obtained in the same way as the estimates leading to (\ref{eqn:largeK}), and so $|\Pi_k\varphi|^\perp_{\alpha,m}\leq 3C_\mu \lambda_0^{(\alpha-\beta)k}$, which combined with (\ref{eqn:stableContr}) proves the estimate for $\|\Pi_k\varphi\circ \Phi^n\|^\perp_{\alpha,m}$.
        \end{proof}
        \begin{proposition}[Lasota-Yorke inequalities]
          \label{prop:L-Y}
          For each $r,n\in\mathbb{N}$, $\alpha,\alpha'$ with $\alpha>\alpha'>0$, $\nu\in (\lambda_0^{-r},1)$ and $m\in\{0,\dots, d\}$, there are $D,E>0$ such that
          $$\vertiii{\mathcal{L} \eta}_{r,\alpha,m}\leq \vertiii{\eta}_{r,\alpha,m}\hspace{.5in}\mbox{ and }\hspace{.5in} \vertiii{\mathcal{L}^n \eta}_{r+1,\alpha',m}\leq D\nu^n \vertiii{\eta}_{r+1,\alpha',m}+ E\vertiii{\eta}_{r,\alpha,m}.$$
        \end{proposition}
        \begin{proof}
          First observe that
          \begin{equation}
            \label{eqn:contraction}
            \int_{\Omega}\langle\varphi, \partial_{i_1}\cdots \partial_{i_p}\mathcal{L}\eta\rangle\, d\mu = \prod_{j=1}^p\lambda_{i_j}^{-1}\int_{\Omega}\langle \varphi \circ\Phi,  \partial_{i_1}\cdots \partial_{i_p} \eta\rangle\, d\mu.
          \end{equation}
          This combined with (\ref{eqn:stableContr}) yields the first inequality. To address the second, first note that
          $$\vertiii{\mathcal{L}^n\eta}_{r+1,\alpha',m} = \sup_{i_1,\dots, i_{r+1}}\sup_{\varphi\in B_{1,m}^\perp(\alpha'+r+1)}\left|\int_\Omega \langle \varphi, \partial^i\mathcal{L}^n\eta\rangle\, d\mu\right|+\vertiii{\mathcal{L}^n\eta}_{r,\alpha',m},$$
          and so (\ref{eqn:contraction}) implies that
          \begin{equation}
            \label{eqn:LYmid}
            \vertiii{\mathcal{L}^n \eta}_{r+1,\alpha',m}\leq \lambda^{-(r+1)n}_0\vertiii{\eta}_{r+1,\alpha',m}+ \vertiii{\mathcal{L}^n\eta}_{r,\alpha',m}.
          \end{equation}

          Now, if $\|\varphi\|^\perp_{\alpha'+p,m}\leq 1$ and $|i|=p\leq r$, again by (\ref{eqn:contraction}):
          \begin{equation*}
            \begin{split}
              \left|\int_\Omega \langle \varphi, \partial^i\mathcal{L}^n\eta\rangle\, d\mu\right| &\leq \lambda_0^{-np}\left|\int_\Omega\langle\varphi\circ \Phi^n, \partial^i \eta\rangle\, d\mu\right| \\
              &\leq \lambda_0^{-pn}\left( \left|\int_\Omega \langle\Delta_k\varphi\circ \Phi^n, \partial^i \eta\rangle \, d\mu\right|+\left|\int_\Omega \langle \Pi_k\varphi\circ \Phi^n,\partial^i \eta\rangle \, d\mu\right|\right)
            \end{split}
          \end{equation*}
          and so using Lemma \ref{lem:repeat},
          \begin{equation*}
            \begin{split}
              \vertiii{\mathcal{L}^n\eta}_{r,\alpha',m}&\leq 4C_\mu \lambda_0^{-rn}\left(\lambda_0^{-(\alpha'+r) k}+\lambda_0^{\Lambda k -(\alpha'+r) n}\right)\vertiii{\eta}_{r,\alpha',m}\\ & \hspace{1in}+4C_\mu \lambda_0^{-rn}\left(1+2 C_\mu \lambda_0^{(\alpha-\alpha')k-n(\alpha+r)}\right)\vertiii{\eta}_{r,\alpha,m}.
            \end{split}
          \end{equation*}
          Combining these estimates with (\ref{eqn:LYmid}):
          \begin{equation*}
            \begin{split}
              \vertiii{\mathcal{L}^n \eta}_{r+1,\alpha',m}&\leq 4C_\mu \lambda_0^{-rn}\left(\lambda_0+\lambda_0^{-(\alpha'+r) k}+\lambda_0^{\Lambda k -(\alpha'+r) n}\right)\vertiii{\eta}_{r+1,\alpha',m}\\ & \hspace{1.5in}+4C_\mu \lambda_0^{-rn}\left(1+2 C_\mu \lambda_0^{(\alpha-\alpha')k-n(\alpha+r)}\right)\vertiii{\eta}_{r,\alpha,m}
            \end{split}
          \end{equation*}
          from which the result follows. 
        \end{proof}
        \subsection{The spectrum of $\mathcal{L}$}
        Propositions \ref{prop:compact} and \ref{prop:L-Y}, combined with Hennion's theorem \cite[Theorem B.14]{DKL}, yields the following.
        \begin{corollary}
          \label{cor:hennion}
          For $\alpha>\Lambda$ and $r\in\mathbb{N}$, the spectrum of $\mathcal{L}:\mathcal{B}^{r,\alpha}_m\rightarrow \mathcal{B}^{r,\alpha}_m$ is contained in the closed unit ball in $\mathbb{C}$ and the essential spectrum is contained in the closed ball of radius $\lambda_0^{-r}$ in $\mathbb{C}$. 
        \end{corollary}
        The following Proposition is a consequence of a theorem of Baladi-Tsuji and shows that to a certain extent the spectrum is independent of Banach spaces used.
        \begin{proposition}
          The discrete spectrum is independent of the Banach space $\mathcal{B}_0^{r,\alpha}$: for $\alpha>\Lambda$ and $r'>r>0$, then the discrete part of the spectrum of $\mathcal{L}$ of norm greater than $\lambda_0^{-r}$ coincides for $\mathcal{L}|_{\mathcal{B}_0^{r,\alpha}}$ and $\mathcal{L}|_{\mathcal{B}_0^{r',\alpha}}$. In addition, the corresponding generalized eigenspaces are contained in $\mathcal{B}_0^{r,\alpha}\cap \mathcal{B}_0^{r',\alpha}$.
        \end{proposition}
        \begin{proof}
          The results will follow from \cite[Lemma A.1]{BT:brin} as long as the inclusion $\mathcal{S}_\alpha^{r'}\subset \mathcal{S}^r_\alpha$ is shown to be dense with respect to $\vertiii{\cdot}_{r,\alpha,m}$.

          Write $f = \sum_{k\geq 0}f^{(k)}\in\mathcal{S}^r_\alpha$ in the canonical way where $f^{(k)} = \pi^*_k g_k$ for some $g_k\in C^r(\Gamma_k)$ as described in \S \ref{subsec:functions}. Since smooth functions are dense in $C^r(\Gamma_k)$, for every $n\in\mathbb{N}$ and $k\in\mathbb{N}_0$ pick a $g_{n,k}\in C^{r'}(\Gamma_k)$ such that $\|g_k-g_{n,k}\|_{C^r(\Gamma_k)}\leq \lambda_0^{-(n+\alpha k)}$. Let $f_n:= \sum_{k\geq 0} \pi^*_kg_{n,k}$. It now needs to be shown that $f_n\rightarrow f$ in $\mathcal{B}_0^{r,\alpha}$.

          If $\varphi\in L^1$ and $\|\varphi\|_\infty\leq 1$,
          \begin{equation}
            \begin{split}
              \left| \int_\Omega \varphi \partial^i(f-f_n)\, d\mu\right| &\leq \sum_{k\geq 0}\left\|\partial^i(\pi^*_k(g_k-g_{n,k})) \right\|_{\infty} \leq \sum_{k\geq 0} \lambda_0^{-i k}\left\|\pi^*_k \partial^i(g_k-g_{n,k}') \right\|_{\infty}  \\
              &\leq \lambda_0^{-n}\sum_{k\geq 0} \lambda_0^{-i k} \lambda_0^{- \alpha k},
            \end{split}
          \end{equation}
          and thus $\vertiii{f-f_n}_{r,\alpha,0}\rightarrow 0$ and the statement follows from \cite[Lemma A.1]{BT:brin}.
        \end{proof}
        Denote by $\Sigma_m^+$ the spectrum of $\Phi^{*}:\check H^m(\Omega;\mathbb{R})\rightarrow \check H^m(\Omega;\mathbb{R})$ consisting of expanding eigenvalues, $\Sigma_m^-$ the spectrum of $\Phi^{-1*}:\check H^m(\Omega;\mathbb{R})\rightarrow \check H^m(\Omega;\mathbb{R})$ consisting of contracting eigenvalues, and set
        $$\sigma_{r,\alpha,m}:= \mathrm{spec}(\mathcal{L}|_{\mathcal{B}^{r,\alpha}_m})\cap \{z\in \mathbb{C}: |z|>\lambda_0^{-r}\}$$
        to be the discrete spectrum of $\mathcal{L}$ on $\mathcal{B}^{r,\alpha}_m$, assuming $r\in\mathbb{N}$ and $\alpha>\Lambda$ by Corollary \ref{cor:hennion}. Note that the map $x\mapsto x^{-1}$ gives a bijection between $\Sigma_m^+$ and $\Sigma_m^-$. 
        
        \begin{lemma}
          \label{lem:spectrum}
          For $\alpha>\Lambda$, $r\in\mathbb{N}$ and $\sigma_{r,\alpha, m}$ as above:
          \begin{enumerate}
          \item $1\in\sigma_{r,\alpha,0}$. It is the unique eigenvalue of modulus 1 and it has multiplicity one.
          \item $\nu\in\sigma_{r,\alpha,0}$ if and only if $\lambda^{-1}\nu\in \sigma_{r,\alpha,d}$.
          \item For $r\in\mathbb{N}$ and $\alpha>\Lambda+1$, $\Sigma_d^-\setminus \sigma_{r+1,\alpha-1,d-1}\subset \sigma_{r,\alpha,d}$.
          \item $\Sigma_d^-\setminus \{\lambda^{-1}\}\subset \sigma_{r,\alpha, d-1}$.
            \item $\sigma_{r,\alpha, m}\subset \left(\sigma_{r+1,\alpha-1,m-1}\cup\Sigma_m^-\cup \sigma_{r-1,\alpha+1,m+1}\right)$.
          \end{enumerate}
        \end{lemma}
        \begin{proof}
          The map $\Phi$ is topologically mixing \cite[Proposition 3.1]{AP}. Property (i) follows from a standard argument depending on this mixing hypothesis; see \cite[\S 4.1]{demers:intro} or \cite[\S 7.1.1]{baladi:book}. The second item follows from the duality between 0 and $d$ forms given by the Hodge-$\star$ operator.

          For (iii), it needs to be shown that if $\nu\in \Sigma_d^-$ and $\nu\not\in\sigma_{r+1,\alpha-1,d-1}$ then there exists a $\eta\in \Psi^d_{r,\alpha}$ such that $(\mathcal{L}-\nu\cdot\mathrm{Id})^k\eta = 0$ for some $k\geq 0$. Now, if $\nu\in\Sigma_d^-\setminus \sigma_{r+1,\alpha-1,d-1}$ then by Theorem \ref{thm:mainCoh} there exists $\eta\in \Psi^d_{r,\alpha}$ such that $(\mathcal{L}-\nu\cdot\mathrm{Id})^k[\eta] = 0$ for all $k$ large enough and so there is a $\omega_\eta\in \Psi^{d-1}_{r+1,\alpha-1}$ such that $(\mathcal{L}-\nu\cdot\mathrm{Id})^k\eta = d\omega_\eta$. If $\theta = (\mathcal{L}-\nu\cdot\mathrm{Id} )^{-k}\omega_\eta\in\Psi^{d-1}_{r+1,\alpha-1}$, then $\eta' = \eta-d\theta$ satisfies $(\mathcal{L}-\nu\cdot\mathrm{Id})^k\eta' = 0$.

          For (iv), note that the smallest element (in norm) of $\Sigma^-_d$ is $\lambda^{-1}$. Suppose that for some $\nu\in \Sigma^-_d$ with $\lambda^{-1}< \nu $, $\nu\not\in \sigma_{r,\alpha, d-1}$. Then by (iii), $\nu\in\sigma_{r-1,\alpha+1, d}$, which by (ii) means that $\lambda \nu\in\sigma_{r-1,\alpha+1, 0}$. But $|\lambda \nu|>1$, contradicting that the spectral radius is 1. So $\nu \in \sigma_{r,\alpha, d-1}$.
          
          Item (v) is due to Daniele Galli, but the proof is included here for completeness. Let $\nu\in \sigma_{r,\alpha, m}$. Then there exists a $\eta\in \Psi^m_{r,\alpha}$ and $J$ such that $(\mathcal{L}-\nu)^J\eta = 0$. If $\eta$ is not closed, then $(\mathcal{L}-\nu)^Jd\eta = d(\mathcal{L}-\nu)^J\eta = 0$, meaning that $\nu\in\sigma_{r-1,\alpha+1,m+1}$. Now suppose that $\eta$ is closed and not exact. Then $(\mathcal{L}-\nu)^J\eta = 0$ implies that $\nu\in \Sigma_m^-$. Now suppose that $\eta = d\theta\neq 0$ for some $\theta\in \Psi^{m-1}_{r+1,\alpha-1}$. Then $(\mathcal{L}-\nu)^J\eta = (\mathcal{L}-\nu)^Jd\theta =  d(\mathcal{L}-\nu)^J\theta =0 $, that is, either $\nu\in \sigma_{r+1,\alpha-1,m-1}$ or $(\mathcal{L}-\nu)^J\theta$ is closed. Suppose $(\mathcal{L}-\nu)^J\theta$ is closed, set $\omega:=(\mathcal{L}-\nu)^J\theta \in \Psi^{m-1}_{r+1,\alpha-1}$. If $(\mathcal{L}-\nu)^J$ is invertible on closed forms, then $\theta = (\mathcal{L}-\nu)^{-J}\omega$ is closed, and $d\theta = \eta = 0$, which is a contradiction. So $(\mathcal{L}-\nu)^J$ is not invertible on closed forms, meaning that $\nu\in \sigma_{r+1,\alpha-1,m-1}$.
        \end{proof}
        Let $\sigma^{-}$ be the eigenvalues associated to generalized eigenvectors of $\Phi^{-1*}:E^{++}\rightarrow E^{++}$. Note that when $d=1$ this implies that $\Sigma_1^- = \sigma^-$. The following proposition gives Theorem \ref{thm:ruelle}.
        \begin{proposition}
          For $r\in\mathbb{N}$ and $\alpha>\Lambda$:
          \begin{enumerate}
          \item if $d=1$, then the set of eigenvalues for $\mathcal{L}$ acting on $\mathcal{B}_0^{r,\alpha}$ contains $\sigma^-\setminus \{\lambda^{-1}\}$. In addition, if $\nu$ is an eigenvalue in $\mathcal{B}_0^{r,\alpha}$ and $k<\alpha-\Lambda$, then $\lambda^{-1}\nu$ is an eigenvalue in $\mathcal{B}_0^{r+k,\alpha-k}$. It follows that if $\mathcal{F}:= \displaystyle \bigcap_{\alpha>0,r>0}\mathcal{S}^r_\alpha$, then the Ruelle spectrum for functions in $\mathcal{F}$ contains the set of numbers of the form $\lambda^{-k}\nu$ with $\nu\in \sigma^-\setminus \{\lambda^{-1}\}$ and $k\in \mathbb{N}$.
          \item If $d=2$, then the set of eigenvalues for $\mathcal{L}$ acting on $\mathcal{B}_0^{r,\alpha}$ contains $\sigma^-\setminus \{\lambda^{-1}\}$. If $\mathcal{S}^{\infty}_\alpha:= \displaystyle\bigcap_{r>0}\mathcal{S}^r_\alpha$, then the Ruelle spectrum for functions in $\mathcal{S}^{\infty}_\alpha$ contains the set $\sigma^-\setminus \{\lambda^{-1}\}$ .
          \end{enumerate}
        \end{proposition}
        \begin{proof}          
          Both of these are consequences of Lemma \ref{lem:spectrum}. For $d=1$, that $\Sigma_1^-\setminus \{\lambda^{-1}\}$ follows directly from part (iii) of Lemma \ref{lem:spectrum}. Now suppose that $\nu\in\sigma_{r,\alpha,0}$ for some $|\nu|<1$. Then by (i), $\lambda^{-1}\nu\in\sigma_{r,\alpha, 1}$. So $(\mathcal{L}-\lambda^{-1}\nu)^{J_1}\eta = 0$ for some $\eta\in\mathcal{B}_1^{r,\alpha}$. Since $\lambda^{-1}|\nu|< \lambda^{-1}$, $\eta$ has to be exact as $\Sigma_1^-$ is bounded from below by $\lambda^{-1}$. So $\eta = d\theta_1$ (where $\theta_1\in \mathcal{B}_0^{r+1,\alpha-1}$ is uniquely defined up to a closed $0$-form, that is, a constant) and $(\mathcal{L}-\lambda^{-1}\nu)^{J_1}\eta = (\mathcal{L}-\lambda^{-1}\nu)^{J_1}d\theta_1 = d(\mathcal{L}-\lambda^{-1}\nu)^{J_1}\theta_1 =0$. Thus either $(\mathcal{L}-\lambda^{-1}\nu)^{J_1}\theta_1$ is closed or $(\mathcal{L}-\lambda^{-1}\nu)^{J_1}\theta_1=0$. If $(\mathcal{L}-\lambda^{-1}\nu)^{J_1}\theta_1$ is closed, then it is constant, and denote by $c_1:= (\mathcal{L}-\lambda^{-1}\nu)^{J_1}\theta_1$. Letting $\theta_1' = \theta_1 - c_1$, it follows that $(\mathcal{L}-\lambda^{-1}\nu)^{J_1}\theta_1' = 0$ and thus it follows that $\lambda^{-1}\nu\in \sigma_{r+1,\alpha-1,0}$. So we're back where we started and the same argument gives that $\lambda^{-2}\nu\in \sigma_{r+2,\alpha-2,0}$ and so on.

          If $d=2$,
          by parts (ii), (iv) and (v) of Lemma \ref{lem:spectrum} it follows that
          $$\Sigma_2^-\setminus\{\lambda^{-1}\}\subset \sigma_{r,\alpha, 1}\subset \left(\sigma_{r+1,\alpha-1,0}\cup \Sigma^-_1\cup \lambda^{-1}\cdot\sigma_{r-1,\alpha+1,0}  \right).$$
          Let $\nu\in \sigma^-\setminus\{\lambda^{-1}\}\subset \Sigma_2^-\setminus\{\lambda^{-1}\}$.
          First, $\nu\not \in \Sigma^-_1$, since by definition $\nu$ contracts faster than the smallest contracting eigenvalue in $\Sigma_1^-$. Now, $\lambda|\nu|>1$, so by (ii) $\nu\not\in \sigma_{r-1,\alpha+1,2} = \lambda^{-1}\sigma_{r-1,\alpha+1,0}$. So it follows that $\nu\in \sigma_{r-1,\alpha+1,0}$.
        \end{proof}

        \section{Applications to primitive substitution subshifts}
        \label{sec:substitution}
        Let $\mathcal{A}$ be a finite set (the alphabet) and $\mathcal{A}^*$ be the set of finite words on $\mathcal{A}$. Let $\varrho:\mathcal{A}\rightarrow \mathcal{A}^*$ be a primitive substitution rule. This means there is an $N$ such that for any $a,b\in \mathcal{A}$ the symbol $a$ appears in $\varrho^N(b)$. Without loss of generality (i.e. by passing to a power) we will assume that there is a symbol $a\in\mathcal{A}$ such that $\varrho(a)$ begins with $a$. Let $\bar{a} = \lim_{N\rightarrow \infty}\varrho^N(a) \in \mathcal{A}^\mathbb{N}$ be a fixed point of the substitution and define $X_\varrho$ to be the orbit closure of $\bar{a}$ under the shift map $\sigma:\mathcal{A}^\mathbb{N}\rightarrow \mathcal{A}^\mathbb{N}$. The system $\sigma:X_\varrho\rightarrow X_\varrho$ is a minimal subshift.

      Define the metric on $X_\varrho$ as
      \begin{equation}
        \label{eqn:ShiftMetric}
        d(\bar{b},\bar{c}) = \lambda^{-k(\bar{b},\bar{c})},
      \end{equation}
      for $\bar{b},\bar{c}\in X_\varrho$,  where $k(\bar{b},\bar{c})\in\mathbb{N}$ is the smallest index $i$ so that $c_i\neq b_i$, and $\lambda>0$ is the Perron-Frobenius eigenvalue of the substitution matrix for $\varrho$.
      
      There is an associated solenoid to $\varrho$ constructed as follows. Let $r:X_\varrho\rightarrow \mathbb{R}^+$ be the function defined as $r(\bar{l}) = v_{l_1}$, where $v\in \mathbb{R}^{|\mathcal{A}|}$ is a positive Perron-Frobenius eigenvector for the substitution matrix, and let $\Omega_\varrho$ be the suspension of $X_\varrho$ with roof function $r$. Then there exists a compact 1-dimensional CW complex $\Gamma$ and an cellular affine (outside the zero-cells of $\Gamma$) map $\gamma:\Gamma\rightarrow \Gamma$ such that
      $$\Omega_\varrho \cong \lim_{\leftarrow}\left(\Gamma,\gamma\right).$$
      This is the Anderson-Putnam construction \cite{AP} and the CW complex $\Gamma$ is refered to as the AP-complex. Denote $\Gamma_k = \pi_k(\Omega_\varrho)$ the ``$k^{th}$'' AP complex. The identification above is not only through a homeomorphism, but in fact an isometry. That is, there is a natural inclusion $i:X_\varrho\rightarrow \Omega_\varrho$ which, under the identification above, can be identified with $C^\perp_0(\bar{a})$. This inclusion is an isometry with respect to the metrics (\ref{eqn:metric}) and (\ref{eqn:ShiftMetric}). Denote by $H_\alpha(X_\varrho)$ the space of $\alpha$-H\"older functions on $X_\varrho$ with respect to this metric.

      Let $0<\varepsilon< \min_{l\in\mathcal{A}} |v_l| /4$. With this choice of $\varepsilon$, the $\varepsilon$-neighborhood of $i(X_\varrho) = C^\perp_0(\bar{a})$ has the local coordinates $(t,c)\in (-\varepsilon,\varepsilon)\times X_\varrho$. If $u_\varepsilon:(-\varepsilon,\varepsilon)\rightarrow \mathbb{R}$ is a smooth even bump function with compact support and of integral 1, then for any function $h:X_\varrho\rightarrow \mathbb{R}$ let $h_{\varepsilon}:\Omega_\varrho\rightarrow \mathbb{R}$ be defined as $h_\varepsilon(t,c) = u_\varepsilon(t) h(c)$ for $(t,c)\in (-\varepsilon,\varepsilon)\times X_\varrho$ and zero otherwise.
      \subsection{The cohomological equation for primitive substitution subshifts}
      This section is dedicated to the proof of Theorem \ref{thm:livsic} on the solutions of the cohomological equation for primitive substitution subshifts $\sigma:X_\varrho\rightarrow X_\varrho$. That is, the goal here is to find a solution $u$ to the equation $f = u\circ\sigma - u$ for a given $f$.
      \begin{lemma}
        \label{lem:imerse}
        If $h:X_\varrho\rightarrow \mathbb{R}$ is $\alpha$-H\"older, then $h_\varepsilon \in \mathcal{S}^1_{\alpha-1}$ for $\alpha>1$. If $r\in\mathbb{N}$ and $\alpha>2$, then the cohomology class $[\star h_\varepsilon]\in H^1_{r,\alpha}(\Omega_\varrho)$ is independent of $u_\varepsilon$.
      \end{lemma}
      \begin{proof}
        The goal is to write
        $$h_\varepsilon = \sum_{k\geq 0}\pi^*_k g_k$$
        in the canonical way described in \S \ref{subsec:functions} with the appropriate bounds on $\|g_k\|_{C^1(\Gamma_k)}$. Using the notation of (\ref{eqn:idempotent}), define $h^k_\varepsilon := \Pi_kh_\varepsilon$ which, in the natural coordinates of the $\varepsilon$-neighborhood of $C^\perp_0(\bar{a})$, is defined as
        $$h^k_{\varepsilon}(t,c) = \hat{\nu}_{k,(t,c)}^{-1}\int_{C^\perp_k(t,c)}h_\varepsilon(t,w)\, d\nu_{k,(t,c)} = u_\varepsilon(t)\hat{\nu}_{k,(t,c)}^{-1}\int_{C^\perp_k(c)}h(w)\, d\nu_{k,c}.$$

        By the same calculation as in \S\ref{subsec:aniso}, $\|h^k_\varepsilon - h_\varepsilon\|_{C^0}\leq C_\mu \|u\|_\infty \lambda^{-\alpha k}$. Now, letting $h^{-1}_\varepsilon = 0$, for any $k\geq 0$ define $\delta_k h_\varepsilon = h^k_\varepsilon - h^{k-1}_\varepsilon$. As such,
        $$\sum_{j=0}^k \delta_jh_\varepsilon = h^k_\varepsilon$$
        and, since $\delta_kh_\varepsilon$ only depends on at most the first $k$ coordinates, $\delta_kh_\varepsilon = \pi^*_k g_k$ for some $g_k:\Gamma_k\rightarrow \mathbb{R}$. If $\alpha>1$, then $\sum_{j=0}^k \pi^*_kg_k \rightarrow h_\varepsilon$ pointwise. It remains to prove the $C^1$ bounds for $g_k$.

        Let $P_k = \gamma^{-k}(\pi_0(\bar{a}))\subset \Gamma_k$ be the preimages of $\pi_0(\bar{a})$ under $\gamma^k$. These points can be of one of two types: flat points or branch points. Flat points have a Euclidean neighborhood whereas branch points do not.
        
        Since $h_{\varepsilon}^k$ is transversally locally constant, $h_\varepsilon^k = \pi^*_k H_k$, for some $H_k:\Gamma_k\rightarrow \mathbb{R}$. Note that the task is to obtain $C^1$ bounds for $g_k = H_k - \gamma^*H_{k-1}$.

        The function $H_k$ is supported in the $\lambda^{-k}\varepsilon$-neighborhood of $P_k$ as follows. If $z\in P_k$ is a flat point, then in the $\lambda^{-k}\varepsilon$-neighborhood of $z\in P_k$, after identifying $z$ with $0$ in these coordinates, $H_k(t) = u_\varepsilon(\lambda^kt)h^k_{\varepsilon}(c_z)$, where $c_z\in C^\perp_0(\bar{a})$ is a point in the clopen subset of $X_\varrho$ corresponding to $z$. From this it follows that in these local coordinates
        \begin{equation}
          \label{eqn:expand}
          \begin{split}
            g_k(t) &= u_\varepsilon(\lambda^kt)(h^k_{\varepsilon}(c_z) - h^{k-1}_\varepsilon(c_z))\\
            &=u_\varepsilon(\lambda^kt)\left[ \hat{\nu}_{k,c_z}^{-1}\int_{C^\perp_k(c_z)} h(w)-h(c_z)\, d\nu_{k,c_z} - \hat{\nu}_{k-1,c_z}^{-1}\int_{C^\perp_{k-1}(c_z)} h(w)-h(c_z)\, d\nu_{k,c_z} \right]
          \end{split}
        \end{equation}
        in a $\lambda^{-k}\varepsilon$-neighborhood of $z\in P_k$. Note that since $h$ is $\alpha$-H\"older, by rewriting it as in (\ref{eqn:expand}),
        $$\left| h^k_{\varepsilon}(c_z) - h^{k-1}_\varepsilon(c_z)\right|\leq 2C_h\lambda^{-\alpha(k-1)}$$
        and so $\|g_k\|_{C^0}\leq 2\lambda C_h\|u_\varepsilon\|_\infty \lambda^{-k\alpha}$. Moreover, in the nieghborhood of $z\in P_k$,
        $$g_k' = \lambda^k u'_\varepsilon(\lambda^kt)(h^k_{\varepsilon}(c_z) - h^{k-1}_\varepsilon(c_z))$$
        and so
        $$\|g_k'\|_\infty = \lambda^k \|u'_\varepsilon\|_\infty \left|h^k_{\varepsilon}(c_z) - h^{k-1}_\varepsilon(c_z)\right|\leq C_h \lambda^k \|u'_\varepsilon\|_\infty \lambda^{-\alpha(k-1)} $$
        and so it follows that $\|g_k\|_{C^1} \leq C \lambda^{-k(\alpha-1)}$ for all $k$, and thus $h_\varepsilon \in\mathcal{S}^1_{\alpha-1}$. The case of $z\in P_k$ being a branched point is essentially treated in the same way: (\ref{eqn:expand}) needs to be written carefully to take into consideration the different branches coming out of $z$. Indeed (\ref{eqn:expand}) treats two branches coming out of $z$ in the flat case, and so (\ref{eqn:expand}) can be used to treat the branched case with minor modifications. The details are left to the reader.

        That the cohomology class is independent of $u_\varepsilon$ follows from the fact that the compactly supported de Rham cohomology of the line is one-dimensional in top degree.
      \end{proof}
      \begin{remark}
        Note that the same proof can be modified to show that $h_\varepsilon\in\mathcal{S}^r_{\alpha-r}$ for any $r<\alpha$. However, the focus is on $r=1$ since the transverse H\"older regularity is what needs to be optimized.
      \end{remark}
      \begin{definition}
        $f:X_\varrho\rightarrow \mathbb{R}$ is a \textbf{coboundary} if there exists a $f:X_\varrho\rightarrow \mathbb{R}$ such that $f = g\circ\sigma - g$. For $\alpha>0$ it is a \textbf{$\alpha$-coboundary} if $f$ is $\alpha$-H\"older and $g$ is $(\alpha-2)$-H\"older. The \textbf{$\alpha$-H\"older cohomology $H_\alpha^0(X_\varrho)$} of $X_\varrho$ is the quotient of $H_\alpha(X_\varrho)$ by the equivalence relation $f_1\sim f_2$ is and only if $f_1-f_2$ is a $\alpha$-coboundary.
      \end{definition}
      In what follows, $X$ is the differential operator in the leaf direction.
      \begin{lemma}
        \label{lem:coboundaries}
        For $\alpha>2$ and an $\alpha$-H\"older function $h$ on $X_\varrho$, $h = g\circ \sigma - g$ for some $g\in H_{\alpha-2}(X_\varrho)$ if and only if $h_\varepsilon = X \Theta$ for some $\Theta\in \mathcal{S}^2_{\alpha-2}$.
      \end{lemma}
      \begin{proof}
        Suppose $h_\varepsilon = X \Theta$. By Lemma \ref{lem:imerse}, $h_\varepsilon\in \mathcal{S}^1_{\alpha-1}$, and so $\Theta \in \mathcal{S}^2_{\alpha-2}$ by Proposition \ref{prop:spaceProp}. Since $u_\varepsilon$ has integral one,
        $$h(c) = \int_{-\varepsilon}^\varepsilon h_\varepsilon(t,c)\, dt =  \int_{-\varepsilon}^\varepsilon X\Theta(t,c)\, dt = \Theta(\varepsilon,c) - \Theta(-\varepsilon,c).$$
        Now, since $h_\varepsilon$ is compactly supported in the $\varepsilon$-neighborhood of $i(X_\varrho)$ and $0 = h_{\varepsilon}(\pm\varepsilon,c) = X\Theta(\pm\varepsilon,c)$, $\Theta$ is leafwise constant on the complement of the $\varepsilon$-neighborhood of $i(X_\varrho)$, which implies that $\Theta(\varepsilon,c) = \Theta(-\varepsilon,\sigma(c))$. Defining $g(c) = \Theta(-\varepsilon,c)$, it follows that $h = g\circ \sigma- g$.

        Write $\Theta = \sum_{k\geq 0} \pi^*_k\Theta_k$ in the canonical way described in \S \ref{subsec:functions}, and, in the local coordinates around $i(X_\varrho)$, let $(-\varepsilon,c_1)\in C_k^\perp((-\varepsilon,c_2))\setminus C_{k+1}^\perp((-\varepsilon,c_2))$. For $\beta>0$:
        \begin{equation*}
          \begin{split}
            \frac{|g(c_1)-g(c_2)|}{\lambda^{-\beta k}} &= \lambda^{\beta k}|\Theta(-\varepsilon, c_1) - \Theta(-\varepsilon, c_2)| \leq \lambda^{\beta k}\sum_{n\geq k} |\pi^*_n\Theta_n (-\varepsilon, c_1) - \pi^*_n\Theta_n (-\varepsilon, c_2)| \\
            &\leq \lambda^{\beta n} \sum_{n\geq k} \|\Theta_n\|_{C^2}\leq \lambda^{\beta k} C_\Theta \frac{\lambda^{-k(\alpha-2)}}{1-\lambda^{-(\alpha-2)}}.
          \end{split}
        \end{equation*}
        Picking $\beta = \alpha-2$, it follows that $g$ has finite $\alpha-2$ H\"older norm, and so $g\in H_{\alpha-2}(X_\varrho)$.
        
        Now suppose $h = g\circ \sigma - g$ and consider $h_\varepsilon$ as constructed above. In the $\varepsilon$-neighborhood of $i(X_\varrho)$, define $\Theta(t,c) = g(c)+\int_{-\varepsilon}^t u_\varepsilon(s)h(c)\, ds $ and so close to $i(X_\varrho)$ it holds that $X\Theta = h_\varepsilon$. By construction, and by the fact that $h = g\circ \sigma - g$, this function satisfies $h_\varepsilon = X\Theta$ globally. The details are left to the reader.
      \end{proof}
      Lemmas \ref{lem:imerse} and \ref{lem:coboundaries} imply that the induced map $j:H_\alpha^0(X_\varrho)\rightarrow H^1_{1,\alpha-1}(\Omega_\varrho)\cong \check H^1(\Omega_\varrho;\mathbb{R})$ defined by $j([f]) = [\star f_\varepsilon]$ is injective whenever $\alpha>2$. This implies that $H^0_\alpha(X_\varrho)$ is finite dimensional and concludes the proof of Theorem \ref{thm:livsic}.
      \bibliographystyle{amsalpha}
      \bibliography{biblio}      
\end{document}